\documentclass[reqno]{siamart0216}

\usepackage{amsmath, amsfonts, latexsym}
\usepackage{amssymb,epsfig,graphicx}
\usepackage{color}
\usepackage{hyperref}
\usepackage[textsize=small]{todonotes}
\usepackage{cancel}
\usepackage[T1]{fontenc}  
\usepackage[utf8]{inputenc} 

\usepackage{floatrow}
\usepackage{enumerate}

\usepackage{cancel}

\newsiamthm{claim}{Claim}
\newsiamthm{prop}{Proposition}
\newsiamthm{assumption}{Assumption}
\newsiamremark{rem}{Remark}
\newsiamremark{expl}{Example}
\newsiamremark{hypothesis}{Hypothesis}
\crefname{hypothesis}{Hypothesis}{Hypotheses}

\newcommand{\Eps}{\mathcal E}
\newcommand{\cblue} [1]{{\color{blue}  #1}}
\newcommand{\eps}{\epsilon}

\newcommand{\ma}{\alpha}

\newcommand{\ml}{\lambda}
\newcommand{\ms}{\sigma}
\newcommand{\mt}{\theta}
\newcommand{\mL}{\Lambda}
\newcommand{\mO}{\Omega}

\newcommand{\cS}{{\mathcal S}}

\newcommand{\N} {{\mathbb N}}
\newcommand{\R} {{\mathbb R}}

\newcommand{\conv}{\rightarrow}

\newcommand{\disp}{\displaystyle}

\newcommand{\be}{\begin{eqnarray}}
\newcommand{\ee}{\end{eqnarray}}

\newcommand{\beno}{\begin{eqnarray*}}
\newcommand{\eeno}{\end{eqnarray*}}

\newcommand{\<}{\langle}
\renewcommand{\>}{\rangle}

\newcommand{\barr}[1]{\begin{array}{#1}}
\newcommand{\earr}{\end{array}}

\usepackage{amsopn}
\DeclareMathOperator{\diag}{diag}
\DeclareMathOperator{\tridiag}{tridiag}
\DeclareMathOperator{\trace}{tr}

\newcommand{\COMMENT}[1]{}
\newcommand{\vio}[1]{{\color{violet!50!red}#1}}

\newcommand{\fud}{\frac{1}{2}}

\begin{document}

\title{Stability and convergence of second order backward differentiation schemes for parabolic Hamilton-Jacobi-Bellman equations}

\headers{BDF2 schemes for HJB equations}{O. Bokanowski, A. Picarelli, C. Reisinger}

\author{
Olivier Bokanowski\thanks{
Laboratoire J.-L. Lions, Universit\'e Pierre et Marie Curie 75252 Paris Cedex 05, France,
and UFR de Math\'ematiques, Site Chevaleret, Universit\'e Paris-Diderot, 75205 Paris Cedex, France
(\email{boka@math.univ-paris-diderot.fr}).}
\and
Athena Picarelli\thanks{Mathematical Institute, University of Oxford, Andrew Wiles Building, Woodstock Rd, Oxford OX2 6GG, UK (\email{athena.picarelli@gmail.com}, \email{christoph.reisinger@maths.ox.ac.uk})}
\and
Christoph Reisinger\footnotemark[2]
}

\date{\today}

\maketitle

\begin{abstract}
We study a second order BDF (Backward Differentiation Formula)  scheme for 
the numerical approximation of parabolic HJB (Hamilton-Jacobi-Bellman) equations.
The scheme under consideration is implicit, non-monotone, and second order accurate in time and space.
The lack of monotonicity prevents the use of well-known convergence results for solutions in the viscosity sense.
In this work, we establish rigorous stability results in a general nonlinear setting as well 
as convergence results for some particular cases with additional regularity assumptions.
While most results are presented for one-dimensional, linear parabolic and non-linear HJB equations, 
some results are also extended to multiple dimensions and to Isaacs equations. 
Numerical tests are included to validate the method.
\end {abstract}


\section{Introduction}\label{sec:intro}
This paper provides stability and convergence results for a type of implicit finite difference scheme
for the approximation of nonlinear parabolic equations using backward differentiation formulae (BDF). 

In particular, we consider Hamilton-Jacobi-Bellman (HJB) equations of the following form:
\be\label{eq:HJB}
  v_t(t,x) +\sup_{a\in \mL}\Big\{\mathcal L^a[v](t,x) +r(t,x,a) v +\ell(t,x,a)\Big\}=0,
\ee
where $(t,x)\in [0,T]\times\R^d$, $\mL\subset \R^m$ is a compact set and 
$$
\mathcal L^a[v](t,x)=-\frac{1}{2}\trace[\Sigma(t,x,a)D^2_x v(t,x)] + b(t,x,a)D_x v(t,x)
$$
is a second order differential operator.
Here, $(\Sigma)_{ij}$ is symmetric non-negative definite for all arguments.
Linear parabolic equations, corresponding to the case $|\mL|=1$, are a special case for which more comprehensive results are obtained in the paper.

It is well known that for nonlinear, possibly degenerate equations the appropriate notion of solutions 
to be considered is that of viscosity solutions \cite{CIL92}.
We assume throughout the whole paper the well-posedness of the problem, 
namely the existence and uniqueness of a solution in the viscosity sense.

Under such weak assumptions, convergence of numerical schemes can only be guaranteed if they satisfy certain monotonicity properties, in addition to the more
standard consistency and stability conditions for linear equations \cite{BS91}. This in turn reduces the obtainable consistency order to 1 in the general case \cite{godunov1959difference}. 

On the other hand, in many cases -- especially in non-degenerate ones -- solutions exhibit higher regularity and are amenable to higher order approximations.
The existence of classical solutions and their regularity properties under a strict ellipticity condition have been investigated, for instance, in  \cite{krylov1982boundedly,eva-len-81}.

The higher order of convergence in both space and time of discontinuous Galerkin approximations is demonstrated theoretically and empirically in \cite{smears2016discontinuous}
for sufficiently regular solutions under a Cordes condition for the diffusion matrix, a measure of the ellipticity.
More recently, it was shown empirically in \cite{bokanowski2016high} 
that schemes based on first derivative approximations in time and space based on a second order backward differentiation formula 
(see, e.g., \cite{suli2003introduction}, Section 12.11, for the definition of BDF schemes for ODEs) have good convergence properties. 
In particular, in a non-degenerate controlled diffusion example therein where the second order, non-monotone Crank-Nicolson scheme fails to converge,
the (also non-monotone) BDF2 scheme shows second order convergence.

For constant coefficient parabolic PDEs, the $L^2$-stability and smoothing properties of the BDF scheme are a direct consequence of the strong A-stability of the scheme.
Moreover, \cite{beale09} shows that for the multi-dimensional heat equation the BDF time stepping solution and its first numerical derivative are stable in the maximum norm.
The technique, which is strongly based on estimates for the resolvent of the discrete Laplacian, do not easily extend to variable coefficients or the nonlinear case.

A more general linear parabolic setting is considered in \cite{becker1998second}, where second order convergence is shown for variable timestep using energy techniques.
This result is extended to a semi-linear example in \cite{Emmrich05}; the application to incompressible Navier--Stokes equations has been analysed in \cite{hill2000approximation}.
In \cite{bokanowski-debrabant}, a closely reltated BDF scheme is studied for a diffusion problem with an obstacle term (which includes the American option problem in 
mathematical finance). 


The scheme we propose is constructed by using a second order BDF approximation for the first derivatives in both time and space. 
Combining this with the standard three-point central finite difference for the second spatial derivative in one dimension, 
the scheme is second order consistent by construction.

For this scheme, we establish new stability results in the $H^1$- and $L^2$-norms 
(see Theorems \ref{th:main} and \ref{th:L2stab}, respectively)
for linear parabolic PDEs and {their nonlinear HJB counterpart}.
{These generalize some results of \cite{becker1998second, Emmrich05, bokanowski-debrabant} to more general non-linear situations.}
From this analysis we deduce error bounds
for classical smooth and piecewise smooth solutions (see Theorems \ref{th:errors} and \ref{th:errors_piecewise}).
Extensions of the results to Isaacs equations and the two-dimensional case are also given.

The outline of the paper is as follows.
In Section \ref{sec:main}, we define some specific BDF schemes and state the main results concerning well-posedness and stability in discrete $H^1$- or $L^2$-norms.
In Sections \ref{sec:scheme} and \ref{sec:3}
we prove the main results and give an extension from HJB to Isaacs equations.
In Section \ref{sec:L2}, we give further stability results in the discrete $L^2$-norm, which are weaker in the sense that they hold only for uncontrolled Lipschitz regulary 
diffusion coefficients,
but stronger in the sense that they allow for degenerate diffusion and can be extended to two dimensions.
In Section \ref{sec:error}, we deduce error estimates from the stability results and from the truncation error of the scheme for sufficiently regular solutions.
Section \ref{sec:num} studies carefully two numerical examples, the Eikonal equation and a second order equation with controlled diffusion.
Section \ref{sec:conclusion} concludes.
An appendix contains a proof of the existence of solutions for our schemes.

\section{Definition of the scheme and main result}\label{sec:main}

We focus in the first instance on the one-dimensional equation 
\begin{subequations}\label{eq:HJ-1d}
\be
  & & v_t + \sup_{a \in \Lambda }\bigg(-\frac{1}{2}\ms^2(t,x,a) v_{xx}  + b(t,x,a) v_x + r(t,x,a) v + \ell(t,x,a)\bigg)
    = 0, \nonumber \\
  & & \hspace{7cm}
   \quad t\in[0,T],\ x\in \R,   \label{eq:HJ-1d.a}\\
  & & 
    v(0,x)=v_0(x) \quad x\in \R.
  \label{eq:HJ-1d.b}
\ee
\end{subequations}

It is known (see Theorem A.1 in \cite{BJ07}) that with the following assumptions:
\begin{itemize} 
\item[--] $\Lambda$ is a compact set,
\item[--]
for some $C_0>0$ the functions $\phi\equiv \sigma, b, r, \ell:[0,T]\times\R\times \Lambda\to\R$ and $v_0:\R\to\R$ satisfy for any $t,s\in [0,T]$, $x,y\in\R$, $a\in\Lambda$
\begin{eqnarray*}
& & |v_0(x)|+ |\phi(t,x,a)|\leq C_0, \\
& & |v_0(t, x) - v_0(s, y)|+|\phi(t, x,a) - \phi(s,y,a)|\leq C_0(|x-y|+|t-s|^{1/2}),
\end{eqnarray*}
\end{itemize}
there exists a unique bounded continuous viscosity solution of \eqref{eq:HJ-1d}.

We will make individual assumptions for each result as we go along, but in general assume a unique and continuous solution (e.g.\ to define the classical truncation error).

\subsection{The BDF2 scheme}

\newcommand{\Xmin}{x_{\min}}
\newcommand{\Xmax}{x_{\max}}
For the approximation in the $x$ variable, we will consider the PDE on a truncated domain $\mO:=(\Xmin,\Xmax)$,
where $\Xmin<\Xmax$. 

\if{
We will assume that the solution $v(t,x)$ is known outside of $\mO$ (or that we have the knowledge of a good approximation
of it).
This setting is used in this work because 
we do not want to focus on the boundary approximations. Since we shall propose second order schemes
in the space variable, the stencil around a given node $x_i$ may includes nodes $x_{i\pm1}$ and $x_{i\pm 2}$.
We will assume that no error is coming from the boundary, leaving this technical point to further analysis.

Note that we could have also consider periodic boundary conditions in order to avoid boundary issues.
}\fi

Let $N\in\N^* \equiv \N\backslash \{0\}$ the number of time steps, $\tau:=T/N$ the time step size, and $t_n = n \tau$, $n=1,\ldots, N$.
Let $I\in \N^*$ the number of interior mesh points, 
and define a uniform mesh $(x_i)_{1\leq i\leq I}$ with mesh size $h$ by
$$
  x_i:= \Xmin + ih, \quad i \in \mathbb{I} = \{1,\dots,I\}, \quad \mbox{where}\quad
  h:=\frac{\Xmax-\Xmin}{I+1}.
$$ 

Hereafter, we denote by $u$ a numerical approximation of $v$, the solution of~\eqref{eq:HJB}, i.e. 
$$u^k_i\sim v(t_k,x_i).$$  
For each time step $t_k$, the unkowns are the values $u^k_i$ for $i=1,\dots,I$.

Standard Dirichlet boundary conditions 
use the knowledge of the values at the boundary, $v(t,\Xmin)$ and $v(t,\Xmax)$. 
Here, as a consequence of the size of the stencil for the spatial BDF2 scheme below, we will assume that values at the two left- and right-most mesh points are given,
that is, $v(t,x_j)$ for $j\in \{-1,0\}$ as well as
$j\in \{I+1, I+2\}$ are known (corresponding to the values at the points $(x_{-1},x_0,x_{I+1},x_{I+2})\equiv (\Xmin-h,\Xmin,\Xmax,\Xmax+h)$).\footnote{
In practice, this means that a sufficiently accurate approximation of these ``boundary values'' has to be available.
Boundary approximations with modified schemes are commonly used and
are not the focus of this paper; it is seen in \cite{picarelli2017error} that the use of a lower order scheme in the vicinity of the boundary does
not affect the global provable convergence order.}


We then consider the following scheme, for $k\geq 2$, $i\in \mathbb{I}$, 
\begin{eqnarray}
\label{eq:BDF2}
&& \hspace{0.5 cm} \mathcal S^{(\tau,h)}(t_k,x_i,u^k_i,[u]_i^k) =
\\
  && \hspace{0.6 cm}  \;\; \frac{3 u^{k}_i - 4u^{k-1}_i + u^{k-2}_{i}}{2\tau} 
    +\ \sup_{a\in \mL} \Big\{L^a[u^k](t_k,x_i)  + r(t_{k},x_i,a) u^{k}_i + \ell(t_{k},x_i,a) \Big\} \ = \ 0, 
   \nonumber
\end{eqnarray}
where we denote as usual by $[u]_i^k$ the numerical solution excluding at $(t_k,x_i)$, and
$$
L^a[u](t_k,x_i) := -\frac{1}{2} \sigma^2(t_{k},x_i,a) D^2 u^{}_i
      + b^+(t_k,x_i,a) D^{1,-} u^{}_i 
    - \; b ^-(t_{k},x_i,a) D^{1,+} u^{}_i,
$$
\beno
  D^2 u_i:=\frac{u_{i-1} - 2 u_i + u_{i+1}}{h^2},
\eeno
(the usual second order approximation of $v_{xx}$),
$b^+:=\max(b,0)$ and {$b^-:=\max(-b,0)$}
denote the positive and negative part of $b$, respectively, 
and where a second order left- or right-sided BDF approximation is used for the first derivative in space:
\be\label{eq:spaceBDF}
  D^{1,-} u_i:= \frac{3 u_i - 4 u_{i-1} + u_{i-2}}{2 h}
  \quad \mbox{and} \quad 
  D^{1,+} u_i:= -\bigg(\frac{3 u_i - 4 u_{i+1} + u_{i+2}}{2h}\bigg). 
\ee
Note in particular the implicit form of the scheme \eqref{eq:BDF2}.
The existence of a unique solution of this nonlinear implicit scheme will be addressed later on. 

We will also define the numerical Hamiltonian associated with the scheme:
$$
  H[u](t_k,x_i):= \sup_{a\in \mL} \Big\{L^a[u](t_k,x_i)  + r(t_{k},x_i,a) u_i + \ell(t_{k},x_i,a) \Big\}.
$$

As discussed above, the scheme is completed by the following boundary conditions:
\beno 
  u^k_i:= v(t_k,x_i),\quad \forall i\in \{-1,0\} \cup \{I+1,I+2\}.
\eeno
Since \eqref{eq:BDF2} is a two-step scheme, for  the first time step $k=1$,
$i \in \mathbb{I}$, we use a backward Euler step,
\be
\label{eq:IEscheme}
&& 
  \mathcal S^{(\tau,h)}(t_1,x_i,u^1_i,[u]_i^1) =
  \\
  && \hspace{2 cm}  \;\; 
   \frac{u^{1}_i - u^{0}_i}{\tau} 
    +\ \sup_{a\in \mL} \Big\{L^a[u^1](t_1,x_i)  + r(t_{1},x_i,a) u^{1}_i + \ell(t_{1},x_i,a) \Big\}  =  0,  
    \nonumber
    \label{eq:BDF2-step1}
\ee
and 
\be
  u^0_i= v_0(x_i), \quad i \in \mathbb{I} 
\ee
is given by the initial condition (\ref{eq:HJ-1d.b}).


\begin{rem}
As the backward Euler  step is only used once, it does not affect the overall second order of the scheme.
\end{rem}

\begin{rem}
Most of our results also apply to the scheme obtained by replacing 
the BDF approximation \eqref{eq:spaceBDF} of the drift term by a centred finite difference approximation:
\be\label{eq:spaceC}
{\widetilde D}^{1,\pm} u_i:= \frac{ u_{i+1}-u_{i-1} }{2h}. 
\ee
However, numerical tests (see Section \ref{sec:num1}) 
show that the BDF upwind approximation as in \eqref{eq:spaceBDF} has a  better behaviour in some extreme 
cases where the diffusion vanishes.
We shall give a rigorous stability estimate for the BDF scheme
in the linear case even for possibly vanishing diffusion
(Section~\ref{sec:linear1d}).
\end{rem}  

\if{
\begin{rem}
For computational purposes we may also want to consider the equation in a bounded domain 
$
\overline\Omega=[x_{\min},x_{\max}]. 
$ 
Given $J\in \N$ and defined the space step $h=(x_{\max}-x_{\min})/J$, one has
$
x_i = x_{\min} + i h 
$
and $\mathbb I=\{0,\ldots,J\}$.
In this case, the scheme has to be modified for  $i=1,J-1$, In particular, we will replace $D^{1,+}u$ (resp. $D^{1,-}u$) below with an upwind first order approximation at $i=J-1$ (resp. $i=1$). We can obtain the results presented in the paper also in this framework under some Dirichlet boundary condition. See also Remark \ref{rem:B.C.} 
\end{rem}
}\fi


\subsection{Definitions and main results}

In the remainder of this paper, we prove various stability and convergence results for the scheme \eqref{eq:BDF2}. 
We state in this section the first main well-posedness and stability results.

Let $u$ denote the  solution of \eqref{eq:BDF2}
and let $v$ be the solution of~\eqref{eq:HJB}. 
The error associated with the scheme is then defined by  
$$
  E^k_i:=u^k_i - v(t_k,x_i). 
$$
For any function $\phi$ we will also use the notation $\phi^k_i:=\phi(t_k,x_i)$
as well as 
$\phi^k:=(\phi^k_i)_{1\leq i\leq I}$
and $[\phi]_i^k := (\phi^m_j)_{(j,m)\neq(i,k)}$,
and the error vector at time $t_k$ is defined by
$$ 
  E^k:=(E^k_1, \dots, E^k_I)^T \ = \ u^k - v^k.
$$

The consistency error will be denoted by $\Eps^k(\phi):=(\Eps^k_i(\phi))_{1\leq i\leq I}\in \R^I$ and is defined in the classical way 
as follows, for any smooth enough function $\phi$:
\begin{small}
\be
  \Eps^k_i(\phi):=  
   \mathcal S^{(\tau,h)}(t_k,x_i,\phi^k_i,[\phi]_i^k) -
   \bigg(\phi_t +\sup_{a\in \mL}\Big\{\mathcal L^a[\phi](t_k,x_i) +r(t_k,x_i,a) \phi +\ell(t_k,x_i,a)\Big\}\bigg).
   \nonumber 
   \\
 \label{eq:consist}
\ee
\end{small}
By extension, for the exact solution $v$ of \eqref{eq:HJB}, we will simply define 
\be
 \label{eq:consist_sol}
\Eps^k_i(v):=\mathcal S^{(\tau,h)}(t_k,x_i,v^k_i,[v]^k_i).
\ee
Note that \eqref{eq:consist_sol} is well-defined for any continuous function.

In particular for the scheme \eqref{eq:BDF2} 
it is clear that we have second order consistency in space and time,
that is, 
\be\label{eq:consist-1}
  |\Eps^k_i(\phi)|\leq c_1(\phi)\tau^2 + c_2(\phi) h^2
\ee
for sufficiently regular data $\phi$.%

\if{
More precisely,
let $C^{p,q}$ denote the set of functions which are $p$ times differentiable with respect to $t$ and $q$ times differentiable with respect
to $x$.
Assuming that $\phi\in C^{3,4}$, by Taylor expansion
the constant $c_1(\phi)$ depends on $\phi_{3t}$ (the third derivative of $\phi$ with respect to the time variable),
and the constant $c_2(\phi)$ depends on $\phi_{3x},\phi_{4x}$, locally around the point of interest.
}\fi

Throughout the paper, $A$ will denote the finite difference matrix associated to the second order derivative, i.e.
\be\label{eqdef:A}
 A: = \frac{1}{h^2}
  \begin{pmatrix} 
   \phantom{-} 2  & -1  & 0     &       &  \\
   -1  &  \phantom{-}2  & -1   &  \ddots & \ddots  \\
       &  \phantom{-}-1  & \ddots&\ddots&  0&   \\
       &     & \ddots&\ddots&   -1  \\
       &     &       &  -1    &  \phantom{-}2
  \end{pmatrix}.
\ee
Let $ \<x,y\>_A:=\<x,Ay\>$. 
Then we consider the $A$-norm defined as follows:
\be\label{eq:Anorm}
  |x|_A^2:=  \<x,Ax\> = \sum_{1\leq i\leq I+1} \left(\frac{x_i-x_{i-1}}{h}\right)^2
\ee
(with the convention in \eqref{eq:Anorm} that $x_0=x_{I+1}=0$). 
\COMMENT{\vio{\\ \fbox{AP:} this is true with zero B.C., in general:
$$
h^2 \<x, Ay\> = \sum^{I+1}_{i=1} (x_i-x_{i-1})(y_i-y_{i-1}) - (x_1-x_0)(y_1-y_0)-(x_{I+1}-x_I)(y_{I+1}-y_I)+x_1y_1+x_Iy_I.
$$
}
}
Hence, $\sqrt{h}|x|_A$ approximates the $H^1$ semi-norm in $\Omega$.
Similarly, we will consider later the standard Euclidean norm defined by $ \|x\|^2:=  \<x,x\>$,
such that $\sqrt{h}\|x\|$ approximates the $L^2$-norm.

{Our first result concerns the solvability of the numerical scheme $\mathcal{S}^{(\tau,h)}$
(seen as an equation for $u^k$, with $[u]^k$ given) and is the following.}


\medskip
{\sc Assumption (A1).}
\ \ $\sigma, b$ and $r$ are bounded functions.

\begin{theorem}\label{th:existence}
Let (A1)
and the following CFL condition hold:
\be
  \label{eq:CFL}
  \|b\|_\infty \frac{\tau}{h} < C. 
\ee
Then, for $\tau$ small enough
and $C=3/2$ (resp. $C=1$) there exists a unique solution of the scheme \eqref{eq:BDF2}
 for $k\geq 2$ (resp. $k=1$, for scheme \eqref{eq:IEscheme}). 
\end{theorem}

The scheme is hence well-defined even if $\ms$ vanishes. A uniform ellipticity condition for $\ms$ will be needed for proving the $H^1$ stability of the scheme.


\medskip
{\sc Assumption (A2).} \ \ 
There exists $\eta>0$ such that 
$$
  \inf_{t\in[0,T]} \inf_{x\in\Omega} \inf_{a\in\mL} \ \sigma^2(t,x,a)\ \geq\ \eta.
$$
We provide a relaxation of the ellipticity condition for stability in the Euclidean norm in Section \ref{sec:linear1d}.

Our main stability result is the following.

\begin{theorem}\label{th:main}
Assume (A1), (A2), as well as the CFL condition \eqref{eq:CFL}.
Then
there exists a constant $C\geq 0$ (independent of $\tau$ and $h$) and $\tau_0>0$ such that, for any $\tau\leq\tau_0$,
\be\label{eq:main-errorH1}
\max_{2\leq k\leq N}|E^k|_A^2 
   & \leq & C \Big(|E^0|_A^2 + |E^1|_A^2 +\tau\sum_{2\leq k\leq N} |\Eps^k(v)|_A^2\Big).
\ee
\end{theorem}
The proof of Theorem \ref{th:main} will be the subject of Section~\ref{sec:3}.

\begin{rem}
As a consequence of the stability result and
under further mild regularity assumptions on the boundary data,
we can deduce that 
the scheme \eqref{eq:BDF2}
is $A$-norm bounded: 
\be\label{eq:main-stabH1}
  \max_{2\leq k\leq N}|u^k|_A^2  \leq C,
\ee
where the constant $C$ depends only on $T$ and on the data but not on $\tau$ and $h$.
\end{rem}

The analysis of the controlled case is made complicated by the fact that even if 
the solution to (\ref{eq:HJ-1d}) is classical and the supremum is attained for each $x$ and $t$ (and similarly
for each $i$ and $k$ in (\ref{eq:BDF2})),
we cannot make any assumptions on the regularity of this optimal control as a function of $x$ and $t$
(or $i$ and $k$, respectively).

In certain circumstances, the previous bound holds with the $A$-norm replaced by the Euclidean norm. 
In particular, we consider the following assumption:

\medskip
{\sc Assumption (A3).} \ \ 
The diffusion coefficient is independent of the control, i.e.\ $\sigma\equiv\sigma(t,x)$ and there exists $L\geq 0$ such that 
$$
|\sigma^2(t,x)-\sigma^2(t,y)| \leq L |x-y| \quad\forall x,y\in\Omega, t\in [0,T].
$$

\begin{theorem}\label{th:L2stab}
Assume (A1), (A2), (A3), as well as the CFL condition \eqref{eq:CFL}.
Then
there exists $C\geq 0$ (independent of $\tau$ and $h$) and  $\tau_0>0$ such that, for any $\tau\leq\tau_0$,
\be\label{eq:main-errorL2}
\max_{2\leq k\leq N}\|E^k\|^2 
   & \leq & C \Big(\|E^0\|^2 + \|E^1\|^2 +\tau\sum_{2\leq k\leq N} \|\Eps^k(v)\|^2\Big).
\ee
\end{theorem}

As a consequence, 
error estimates will be obtained under the main assumptions (A1), (A2) and (A3) or under 
some specific assumptions, see Sections \ref{sec:L2} and  \ref{sec:error}.

The extension of the presented results to other type of nonlinear operators 
($\inf$, $\sup\inf$ or $\inf\sup$) 
and corresponding equations will also be discussed.

\section{Proof of Theorem \ref{th:existence} {\bf (well-posedness of the scheme)}}\label{sec:scheme}
The scheme \eqref{eq:BDF2} at time $t_k$ (for $k\geq 2$) can be written in the following form: 
$$
 {\sup}_{a\in \Lambda} ( M^k_a X - q^k_a ) = 0,
$$
where $q^k_a \in \R^I$ and $M^k_a \in \R^{I\times I}$ with the following non-zero entries:
\be
 & & 
  (M^k_a)_{i,i}  :=   \frac{3}{2} + \tau \bigg\{2 \frac{\ms^2}{h^2} + \frac{3 b^+}{2 h} + \frac{3 b^-}{2h} + r \bigg\}
   \\
 & & 
  (M^k_a)_{i,i+1}  :=    \tau \bigg\{-  \frac{\ms^2}{h^2} - \frac{4 b^-}{2 h}  \bigg\},
 \quad
  (M^k_a)_{i,i-1}  :=    \tau \bigg\{-  \frac{\ms^2}{h^2} - \frac{4 b^+}{2 h}  \bigg\}
   \\
 & & 
  (M^k_a)_{i,i+2}  :=    \tau \frac{b^-}{2h}
 \quad
  (M^k_a)_{i,i-2}  :=    \tau \frac{b^+}{2h}
\ee
with $\ms\equiv \ms(t_k,x_i,a)$, $b^\pm\equiv b^\pm(t_k,x_i,a)$ and $r\equiv r(t_k,x_i,a)$.
For $k=1$, the terms are different but the form (and analysis) is similar.
The fact that $(M_a)_{i,i\pm 2}$ are non\-negative breaks the monotonicity of the scheme and makes the analysis more difficult.

We will use the following lemma, whose proof is given in appendix~\ref{app:A}:
\begin{lemma}\label{lem:exist}
Asssume that $\Lambda$ is some set, $(q_a)_{a\in \Lambda}$ is a family of vectors in $\R^I$, 
$(M_a)_{a\in \Lambda} $ is a family of matrices in $\R^{I\times I}$ such that:
\begin{itemize}
\item[$(i)$] 
for all $a\in \Lambda$,
$$ (M_a)_{ii}>0; $$
\item[$(ii)$] (a form of diagonal dominance)
\be \label{eq:lem:exist-2}
  \sup_{a\in \Lambda} \max_{i\in \mathbb{I}} \frac{\sum_{j>i} |(M_a)_{ij}|}{|(M_a)_{ii}| - \sum_{j<i} |(M_a)_{ij}|} < 1.
\ee

\end{itemize}
Then there exists a unique solution $X$ in $\R^n$ of 
\be\label{eq:lem-1}
  {\sup}_{a\in \Lambda} (M_a X - q_a) = 0.
\ee
\end{lemma}

\begin{rem}
For a fixed $a\in \Lambda$, we have
$$ 
 \max_{i\in \mathbb{I}} \frac{\sum_{j>i} |(M_a)_{ij}|}{|(M_a)_{ii}| - \sum_{j<i} |(M_a)_{ij}|} < 1
 \quad 
   \Leftrightarrow 
 \quad 
 \min_{i\in \mathbb{I}} \bigg( |(M_a)_{ii}| - \sum_{j\neq i} |(M_a)_{ij}| \bigg) >0.
$$
{
Moreover, if $\Lambda$ is compact and $a\conv M_a$ is continuous, then
\eqref{eq:lem:exist-2} is equivalent to 
$$
 \inf_{a\in \Lambda} \min_{i\in \mathbb{I}} \bigg( |(M_a)_{ii}| - \sum_{j\neq i} |(M_a)_{ij}| \bigg) >0.
$$
}
\end{rem}

\begin{proof}[Proof of Theorem~\ref{th:existence}]
We are going to prove properties $(i)$ and $(ii)$ in Lemma \ref{lem:exist}. Condition $(i)$ is immediately verified, 
\if{
Next, we estimate
\beno
  \sum_{j\neq i} |(M_a)_{ij}|
    & \leq &  \tau \bigg( 2 \frac{\ms^2}{h^2}  + \frac{4 b^+}{2h} + \frac{4 b^-}{2h} + \frac{1}{2h} (b^- + b^+)\bigg)  \\
    & \leq    & \tau \bigg( 2 \frac{\ms^2}{h^2}  + \frac{5 |b|}{2h} \bigg)
\eeno
(the inequality is strict only for boundary terms),
and 
\beno
  |(M_a)_{ii}|
   =  \frac{3}{2} + \tau \bigg\{2 \frac{\ms^2}{h^2} + \frac{3 |b|}{2 h} + r \bigg\}.
\eeno
Hence 
$$
  |(M_a)_{ii}| -  \sum_{j\neq i} |(M_a)_{ij}| \geq \frac{3}{2} - \frac{\tau \|b\|_\infty}{h}  =: \mu_a
$$ 
with $\mu_a>0$ as soon as the CFL condition \eqref{eq:CFL} is satisfied. Therefore $(ii)$ is verified.
}\fi
and we turn to proving  $(ii)$.
We have 
$$
  \mu_1 := \sum_{j>i} |(M_a)_{ij}| \leq 
     \tau \bigg( \frac{\ms_{i}^2}{h^2}  + \frac{5 b^-_{i}}{2h}\bigg)
$$
(omitting the dependency on $k$ and $a$ in $\ms,b^\pm,r$) and 
$$
  \mu_2 := |(M_a)_{ii}| - \sum_{j<i} |(M_a)_{ij}| \geq 
     \frac{3}{2} +  \tau \bigg( \frac{\ms_{i}^2}{h^2}  
       - \frac{2 b^+_{i}}{2h}  + \frac{3 b^-_{i}}{2h}   + r \bigg).
$$
By the CFL condition \eqref{eq:CFL}, there exists $\eps>0$ such that  
$\frac{\tau \|b\|_\infty}{h} \leq \frac{3}{2}- \eps$.
This implies 
$$
     \frac{3}{2} -\frac{\eps}{2} + \tau \left(-\frac{2 b^+_{i}}{2h}  + \frac{3 b^-_{i}}{2h}\right)
      \geq  \frac{\eps}{2} +  \tau \frac{5b^-_{i}}{2h}
$$
and therefore 
$$ \mu_2 \geq  \bigg(\tau \frac{\ms_{i}^2}{h^2} + \frac{\eps}{2} +  \, \tau r\bigg)  
  +  \bigg(\tau \frac{5 b^-_{i}}{2h}  + \frac{\eps}{2}\bigg).
$$
Then by using 
$\disp \frac{a_1+a_2}{c_1+c_2}\leq \max\Big(\frac{a_1}{c_1}, \frac{a_2}{c_2}\Big)$
for numbers $a_i,c_i\geq 0$, we obtain
$$ \frac{\mu_1}{\mu_2}  \leq \max\bigg(
  \frac{\tau \frac{\ms^2_i}{h^2}}{\tau \frac{\ms^2_i}{h^2} + \frac{\eps}{2}+\ \tau\,r },\
  \frac{\tau \frac{5 b^-_{i}}{2h}}{\tau \frac{5 b^-_{i}}{2h}  + \frac{\eps}{2}}\bigg).
$$
Taking $\tau$ small enough such that for instance
$\frac{\eps}{2} + \tau r \geq \frac{\eps}{4}$,  and since $b(.)$ and $\ms(.)$ are bounded
functions (by (A1)), we obtain the bound 
\beno
  \sup_{a\in A} \max_{i\in \mathbb{I}}
    \frac{\sum_{j>i} |(M_a)_{ij}|}{|(M_a)_{ii}| - \sum_{j<i} |(M_a)_{ij}|} 
    & \! \leq \! &  \max\bigg(
  \frac{\tau \frac{\|\ms^2\|_\infty}{h^2}}{\tau \frac{\|\ms^2\|_\infty}{h^2} 
     + \frac{\eps}{4}},\
  \frac{\tau \frac{5 \|b^-\|_\infty}{2h}}{\tau \frac{5 \|b^-\|_\infty}{2h}
     + \frac{\eps}{2}}\bigg) < 1.
\eeno

Since the last bound is a constant $<1$, 
we can apply Lemma~\ref{lem:exist}
to obtain the existence and uniqueness of the solution of the BDF2 scheme. 
\end{proof}
\noindent

\section{Proof of Theorem~\ref{th:main} {\bf (stability in the $A$-norm)} }\label{sec:3}
The proof consists of three main steps: first, we show a ``linear''
recursion for the error 
(Lemma \ref{lem:nonlin2lin});
second, we pass from such a recursion for the error in vector form to a scalar recursion (Lemma \ref{lem:vect2Anorm});
finally, we show the stability estimate from this scalar recursion 
(Lemma \ref{lem:stability}).

\subsection{Treatment of the nonlinearity}\label{sec:vec2lin}

First, we have the following: 
\begin{lemma}\label{lem:nonlin2lin}
Let $u$ be the solution of scheme \eqref{eq:BDF2} and $v$ 
the solution of equation \eqref{eq:HJ-1d}. 
There exist coefficients $\tilde\sigma^k_i$, $(\tilde b^\pm)^k_i$, $\tilde r^k_i$, such that the error $E^k=u^k-v^k$ satisfies 
\begin{eqnarray}
 \nonumber
 \frac{3 E_i^k -4 E_i^{k-1} + E_i^{k-2} }{2\tau} 
   -\frac{1}{2}(\tilde\ms^2)^k_i D^2 E^{k}_i  + (\tilde b^+)^k_i D^{1,-} E^{k}_i -(\tilde b^-)^k_i D^{1,+} E^{k}_i 
   + \tilde r^k_i E^k_i = -\Eps_i^k \\
 \label{eq:vecrec}
\end{eqnarray}
for any  $k\geq 2$ and $i\in\mathbb I$,
where $(\tilde \ms^2)^k_i$, $(\tilde b^\pm)^k_i$, $\tilde r^k_i$
belong, respectively, to the convex hulls $co(\ms^2(t_k,x_i,\Lambda))$, $co(b^\pm(t_k,x_i,\Lambda))$, $co(r(t_k,x_i,\Lambda)$.
\end{lemma}
\begin{proof}
By definition of the consistency error \eqref{eq:consist_sol}, 
one has (for $k\geq 2$, $1\leq i\leq I$)
\be\label{eq:vBDF}
  \frac{3 v^k_i -4 v^{k-1}_{i} + v^{k-2}_i}{2\tau} +  H[v^k](t_k,x_i)  = \Eps^k_i.
\ee
The scheme simply reads
\be\label{eq:uBDF}
  \frac{3 u^k_i -4 u^{k-1}_{i} + u^{k-2}_i}{2\tau} +  H[u^k](t_k,x_i)  =  0.
\ee
Subtracting \eqref{eq:vBDF} from \eqref{eq:uBDF}, denoting also
$H[u^k]\equiv (H[u^k](t_k,x_i))_{1\leq i\leq I}$, 
the following recursion is obtained for the error in $\R^I$:
\be\label{eq:err-recu}
  \frac{3 E^k -4 E^{k-1} + E^{k-2} }{2\tau} +  H[u^k] - H[v^k]  = -\Eps^k.
\ee

For simplicity of presentation, we first consider the case when $b$ and $r$ vanish, i.e. $
  \mbox{$b(.)\equiv 0$ and $r(.)\equiv 0$.}
$
In this case, 
\be\label{eq:Huk}
 H[u^k]_i = \sup_{a\in \mL} \Big\{- \frac{1}{2} \ms^2(t_k,x_i,a) (D^2 u^k)_i  + \ell(t_k,x_i,a)\Big\}.
\ee
To simplify the presentation, we will assume that $\ms$ and $\ell$ are continuous functions of $a$ so that the supremum is attained.\footnote{The general case is obtained easily by considering sequences of $\epsilon$-optimal controls and letting $\epsilon\rightarrow 0$, such that (\ref{eq:Hgen}) below still holds for a suitably defined 
$\tilde{\sigma}^{_2}$, $\tilde{b}^{^{_+}}$, $\tilde{b}^{^{_-}}$, $\tilde r$.}
For each given $k,i$, let then $\bar a^k_i \in \mL$ denote an optimal control in \eqref{eq:Huk}.
\\
In the same way, let ${\bar b}^k_i$ denote an optimal control for $H[v^k]_i$. By using the optimality of $\bar a^k_i$, it holds
\begin{align}
& H[u^k]_i  - H[v^k]_i \nonumber\\
& =\! - \frac{1}{2} \ms^2(t_k,x_i, \bar a^k_i)  (D^2 u^k)_i \! + \ell(t_{k},x_i,\bar a^k_i) - \sup_{a\in\Lambda}\Big\{\!\! -\! \frac{1}{2} \ms^2(t_k,x_i, a)  (D^2 v^k)_i \! + \ell(t_{k},x_i,a)\!\Big\}  \nonumber \\
& \leq\! - \frac{1}{2} \ms^2(t_k,x_i, \bar a^k_i)  (D^2 u^k)_i 
         - \Big(- \frac{1}{2} \ms^2(t_k,x_i, \bar a^k_i)  (D^2 v^k)_i  \Big)  \nonumber \\
  &  =\! - \frac{1}{2} \ms^2(t_k,x_i, \bar a^k_i) (D^2 E^k)_i \label{eq:err01}
\end{align}
and, in the same way, 
\be
  H[u^k]_i  - H[v^k]_i \geq - \frac{1}{2} \ms^2(t_k,x_i,\bar b^k_i) (D^2 E^k)_i.
  \label{eq:err02}
\ee
Therefore, combining \eqref{eq:err01} and \eqref{eq:err02}, 
$H[u^k]_i  - H[v^k]_i$ is a convex combination of $-\frac{1}{2} \ms^2(t_k,x_i,\bar a^k_i) (D^2 E^k)_i$ and 
 $-\frac{1}{2} \ms^2(t_k,x_i,\bar b^k_i) (D^2 E^k)_i$. 
In particular, we can write
\be\label{eq:Hdiff}
  H[u^k]_i  - H[v^k]_i = -\frac{1}{2} {\tilde\ms}^2 (t_{k},x_i) (D^2 E^k)_i,
\ee
where $\tilde\ms^{2}(t_{k},x_i)$ is a convex combination of $\ms^2(t_k,x_i,\bar a^k_i)$ and $\ms^2(t_k,x_i,\bar b^k_i)$.\\
In the general case (i.e. $b,r\not\equiv 0$)
one gets similarly
\be\label{eq:Hgen}
  H[u^k]_i  - H[v^k]_i = -\frac{1}{2}(\tilde\ms^2)^k_i D^2 E^{k}_i  + ({\tilde b}^+)^k_i D^{1,-} E^{k}_i 
    - (\tilde b^-)^k_i D^{1,+} E^{k}_i +\tilde r^k_i E^k_i,
\ee
where, for $\phi=\sigma^2,b,r$,
$$
\tilde\phi^k_i := \gamma^k_i \phi(t_{k},x_i,\bar a^k_i) +(1-\gamma^k_i) \phi(t_{k},x_i,\bar b^k_i)
$$
for some $\gamma^k_i\in[0,1]$. 
\end{proof}

\subsection{Isaacs equations}\label{sec:3.4}

The same technique used above to deal with the nonlinear operator applies also to Isaacs equations, i.e.\ equations of the following form:
\be
v_t +\sup_{a\in \Lambda_1}\inf_{b\in\Lambda_2}\Big\{-\mathcal L^{(a,b)}[v](t,x) +r(t,x,a,b) v +\ell(t,x,a,b)\Big\}=0,
\ee
where $(t,x)\in [0,T]\times\R^d$, $\Lambda_1,\Lambda_2\subset \R^m$ are compact sets and 
$$
\mathcal L^{(a,b)}[v](t,x)=\frac{1}{2}\ms^2(t,x,a,b)v_{xx} + b(t,x,a,b)v_x.
$$


To simplify the presentation, let us consider again $b,r \equiv 0$, and now also $\ell \equiv 0$. By analogous definitions and reasoning to above, we get \eqref{eq:err-recu}, where, for $\phi=u,v$,
\be\label{eq:HIsaacs}
H[\phi^k]_i = \sup_{a\in \Lambda_1}\inf_{b\in\Lambda_2}\Big\{-\frac{1}{2}\ms^2(t,x,a,b)(D^2_x \phi^k)_i\Big\}.
\ee
Let  $(\bar a^k_i,\bar b^k_i)\in\Lambda_1\times\Lambda_2$ denote an optimal control in  \eqref{eq:HIsaacs}.\footnote{Or, if not attained,
use an approximation argument.}
One has  
$$
H[u^k]_i = \sup_{a\in \Lambda_1}\Big\{-\frac{1}{2}\ms^2(t,x,a,\bar b^k_i)(D^2_x u^k)_i\Big\}=\inf_{b\in \Lambda_2}\Big\{-\frac{1}{2}\ms^2(t,x,\bar a^k_i,b)(D^2_x v^k)_i\Big\}.
$$
Therefore 
\begin{align}
& H[u^k]_i  - H[v^k]_i \nonumber\\
& =\ \sup_{a\in \Lambda_1}\Big\{-\frac{1}{2}\ms^2(t,x,a,\bar b^k_i)(D^2_x u^k)_i\Big\} - \sup_{a\in\Lambda_1}\inf_{b\in\Lambda_2}\Big\{ - \frac{1}{2} \ms^2(t_k,x_i, a,b)  (D^2 v^k)_i \Big\}  \nonumber \\
& \geq\ \sup_{a\in \Lambda_1}\Big\{-\frac{1}{2}\ms^2(t,x,a,\bar b^k_i)(D^2_x u^k)_i\Big\} - \sup_{a\in\Lambda_1}\Big\{ - \frac{1}{2} \ms^2(t_k,x_i, a,\bar b^k_i)  (D^2 v^k)_i \Big\}  \nonumber\\
&   \geq \inf_{a\in \Lambda_1}\Big\{-\frac{1}{2}\ms^2(t,x,a,\bar b^k_i)(D^2_x E^k)_i\Big\}. \label{eq:IsaacsInf}
\end{align}
Analogously, one can prove
\begin{align}\label{eq:IsaacsSup}
 H[u^k]_i  - H[v^k]_i \leq \sup_{b\in \Lambda_2}\Big\{-\frac{1}{2}\ms^2(t,x,\bar a^k_i,b)(D^2_x E^k)_i\Big\}
\end{align}
(here, we also use $\inf(a)-\inf(b)\leq \sup(a-b)$ and $\sup(a)-\sup(b)\geq \inf(a-b)$).
At this point, it is sufficient to take for $\hat a^k_i\in\Lambda_1$ and $\hat b^k_i\in\Lambda_2$ optimal controls in \eqref{eq:IsaacsInf} and \eqref{eq:IsaacsSup}, respectively, to be able to write $H[u^k]_i  - H[v^k]_i$ as a convex combination
of $-\frac{1}{2}\ms^2(t,x,\hat a^k_i,\bar b^k_i)(D^2_x E^k)_i$ and $-\frac{1}{2}\ms^2(t,x,\bar a^k_i,\hat b^k_i)(D^2_x E^k)_i$.

From this, an equation exactly as in \eqref{eq:vecrec} can be derived, with a suitable convex combination $(\tilde{\sigma^{_2}})_i^k$ of diffusion coefficients, and similar for the drift and other terms.

\subsection{A scalar error recursion}\label{sec:vect2scalar}
From \eqref{eq:vecrec}, we can derive the following:
\begin{lemma}\label{lem:vect2Anorm}
Let assumptions (A1) and (A2) in Theorem \ref{th:main} be satisfied. 
Then there exists a constant $C\geq 0$ such that
\be 
  & &
  \hspace{-1.0cm}
  \frac{1}{2} \Big( (3-C\tau)|E^k|_A^2 -4| E^{k-1}|_A^2 +|E^{k-2}|_A^2\Big)
    +|E^k-E^{k-1}|_A^2-|E^{k-1}-E^{k-2}|_A^2
   \nonumber \\
  & & \hspace{8.3cm} \leq 2 \tau |E^k|_A\;|\Eps^k|_A.  \label{eq:one_step1}
\ee
\end{lemma}

\begin{proof}
For simplicity of presentation we will assume that $b$ has constant positive sign. The terms coming from the negative part of $b$ can 
be treated in a similar way.
\\
We remark that for $E\in \R^I$,
$ - D^2 E = A E,$
where $A$ is the finite difference matrix defined in \eqref{eqdef:A}.
By \eqref{eq:vecrec}, we get the following: 
\be\label{eq:err-recu-lin_gen}
  \frac{3 E^k -4 E^{k-1} + E^{k-2} }{2\tau} +  \Delta^k A E^k + F^k B E^k + R^k E^k = -\Eps^k,
\ee
where $\Delta^k:=\frac{1}{2} \diag( (\tilde\ms^2)^k_i),\
F^k  =  \diag(\tilde b^k_i),\ R_k  =  \diag(\tilde r^k_i)$ and
\begin{align*}
B  = \frac{1}{2h}
  \begin{pmatrix} 
    \phantom{-}3 & \phantom{-}0  &     &        &  \\
    -4  &  \phantom{-}3  &  0 &   &  \\
     \phantom{-}1 &  -4  & \ddots & \ddots &   &   \\
    \phantom{-}0 &\ddots& \ddots & \ddots&  0  \\
     \ddots  &   \ddots  &   1    &  -4   &  3   
\end{pmatrix}.
\end{align*}
We form the scalar product of \eqref{eq:err-recu-lin_gen} with $AE^k$. 
By using the identity
$
  2 \langle a-b,a\rangle_A=|a|_A^2+|a-b|_A^2-|b|_A^2,
$
one has:
\be
  & & \<3 E^k -4 E^{k-1} + E^{k-2}, E^k\>_A \nonumber\\
  & & \ =  4\<E^k-E^{k-1},E^k\>_A - \<E^{k}-E^{k-2},E^k\>_A \nonumber \\
  & & \ =  \frac{1}{2} \left(4 |E^k|_A^2 + 4 |E^k-E^{k-1}|_A^2 - 4|E^{k-1}|_A^2\right) 
     - \frac{1}{2}\left(|E^{k}|_A^2 + |E^{k}-E^{k-2}|_A^2 -|E^{k-2}|_A^2\right) \nonumber\\
  & & \ \geq \frac{1}{2} \left( 3 |E^k|_A^2  - 4|E^{k-1}|_A^2 + |E^{k-2}|_A^2\right) + |E^k-E^{k-1}|_A^2 - |E^{k-1}-E^{k-2}|_A^2, \label{eq:time-est}
\ee
where we have also used $|a+b|^2\leq 2|a|^2 + 2|b|^2$. 
From $(\sigma^2)^k_i\geq \eta >0$ for all $k,i$:
\be\label{eq:est_diff}
  \langle \Delta^k A E^k, A E^k\rangle \geq \frac{\eta}{2} \| A E^k \|^2,
\ee
where $\|\cdot\|$ denotes the canonical Euclidean norm in $\R^I$.



In order to estimate the drift component, let us introduce the notation
\begin{equation}\label{eqdef:deltaE}
  \delta E:=(E_i-E_{i-1})_{1\leq i\leq {I}}
\end{equation}
with the convention that $E_i=0$ for all indices $i$ which 
are not in ${\mathbb{I}}$. It holds:
\beno
| \<F^k B E^k, A E^k\> |
  & = & \bigg| \frac{1}{2h}\langle F^k (3E^k_i - 4 E^k_{i-1} +E^k_{i-2})_{i\in\mathbb I}, A E^k\rangle \bigg| \\
  & = & \bigg| \frac{1}{2h}\langle F^k \big(3 \delta E^k -  \delta^2 E^k\big),\ A E^k\> \bigg| \\
  & \leq &  \frac{1}{2h}\Big\{3\| F^k  \delta E^k \|\, \| A E^k \| + \|F^k  \delta^2 E^k\|\, \|A E^k\| \Big\}.
\eeno
By using the boundedness of the drift term, and 
$\|\delta E^k\|,\|\delta^2 E^k\| \leq h |E^k|_A$,
\be
 | \<F^k B E^k, A E^k\> |
  & \leq & \frac{\|b\|_\infty}{2h}\Big\{3\|A E^k\|\| \delta E^k\| + \|A E^k\|\|  \delta^2 E^k\|\Big\}\nonumber\\
  & \leq &   2 \|b\|_\infty\ \|A E^k\|\, |E^k|_A. \label{eq:est_drift}
\ee
\COMMENT{
\vio{\fbox{AP:} without zero B.C. and defining $|x|_A^2=(x,Ax)$ we would have
\begin{align*}
\|\delta e\|^2 & = h^2\<e, Ae\>+(e_1-e_0)^2-e_1^2-e_I^2\\
\|\delta e^2\|^2 & = h^2\<e, Ae\>+2(e_1-e_0)^2-(e_I-e_{I-1})^2-e_1^2-e_I^2
\end{align*}
}
}


For the last term, using the boundedness of $r$ and the Cauchy-Schwarz inequality,
\be\label{eq:estR}
  |\langle R^k E^k, A E^k\rangle| \leq \|r\|_\infty \|E^k\| \|A E^k\|.
\ee

Therefore, putting \eqref{eq:est_diff}, \eqref{eq:est_drift} and \eqref{eq:estR}  together,
\be
  & & \hspace{-2cm} \<\Delta^k A E^k + F^k B E^k + R^k E^k ,AE^k\> \nonumber \\
  & \geq & \frac{\eta}{2} \|A E^k\|^2 - 2 \|b\|_\infty\|A E^k\|\, |E^k|_A 
     - \|r\|_\infty \|A E^k\| \|E^k\| . \label{eq:bound2}
\ee
Easy calculus shows that the minimal eigenvalue of $A$ is $\ml_{\min}(A)=\frac{4}{h^2} \sin^2(\frac{\pi h}{2})\geq~4$.
Hence $\<X,AX\> \geq 4 \<X,X\>$ and therefore $\|X\|\leq \frac{1}{2} |X|_A$.
In the same way, we have also $|X|_A \leq \frac{1}{2} \|AX\|$.
Hence 
it holds 
\be
  \<\Delta^k A E^k + F^k B E^k + R^k E^k ,AE^k\> 
   \geq \frac{\eta}{2} \|A E^k\|^2 - C_1 \| A E^k\| |E^k|_A 
  \label{eq:bis-1}
\ee
with $C_1:=2 \|b\|_\infty + \frac{1}{2} \|r\|_\infty$.
By using 
$  C_1 \| A E^k\| |E^k|_A \leq \frac{\eta}{2} \|A E^k\|^2 + \frac{1}{2\eta} C_1^2 |E^k|_A^2$,
\be
  \<\Delta^k A E^k + F^k B E^k + R^k E^k ,AE^k\> 
 \geq  - \frac{1}{2\eta} C_1^2 |E^k|^2_A.
  \label{eq:bis-2}
\ee
Then, combining \eqref{eq:time-est} and \eqref{eq:bis-2}, we obtain the desired inequality with 
$C:=\frac{2}{\eta} C_1^2$.
\end{proof}

\COMMENT{
\vio{
\fbox{AP:} In the study of the stability (i.e. no zero B.C.) we would have:
\beno
  \<\Delta^k A E^k + F^k B E^k + R^k E^k ,AE^k\> 
 & \geq &  \frac{\eta}{2} \|A E^k\|^2 - \frac{C_1}{h} \| A E^k\| \|\delta E^k\| -\frac{C_1}{h} \| A E^k\|\|\delta^2 E^k\|\\
& \geq &  - \frac{C_1^2}{4\eta }\frac{ \|\delta E^k\|^2}{h^2} -\frac{C_1^2}{4\eta }\frac{ \|\delta^2 E^k\|^2}{h^2}\\
& \geq & - \frac{C_1^2}{2\eta}|E^k|^2_A - \frac{3C_1^2}{2\eta}\frac{(e_1-e_0)^2}{h^2}\quad\Bigg(\approx - \frac{C_1^2}{2\eta}|E^k|^2_A - \frac{3C_1^2}{2\eta}(e_x)_0^2\Bigg)
\eeno
}
}

%
%
\if{
\begin{rem}
In the case when $r(.)$ is nonconstant, by using the boundedness of $r$ and the Cauchy-Schwarz inequality,
we obtain
\be\label{eq:estR-b}
  |\langle R^k E^k, A E^k\rangle| \leq \|r\|_\infty \|E^k\| \|A E^k\|.
\ee
Then, by applying the Kantorovith inequality
$$
  \<y, A^{-1} y\> \<y, A y\> \leq \frac{1}{4}\bigg(\sqrt{\frac{\ml_{max}(A)}{\ml_{min}(A)}} +  \sqrt{\frac{\ml_{min}(A)}{\ml_{max}(A)}} \bigg)
  \| y\|^4
$$
to  $y:=A^{1/2} X$, and by using the fact that $\ml_{min}(A)\geq 4$ and that $\ml_{max}(A)\leq \frac{4}{h^2}$, we obtain
$$
  \|X\|^2\ \|A X\|^2  \leq \frac{1}{4} (\frac{1}{h}  + 1) |X|_A^4.
$$
Hence there exists a constant $C\geq 0$ independent of $h$ such that, for all $X\in \R^n$: 
\beno
  \|X\| \| A X\| \leq \frac{C}{\sqrt{h}} |X|_A^2
\eeno
We therefore obtain the following inequality:
\be\label{eq:estR-2-b}
  |\< R^k E^k, A E^k\>| \leq \frac{C \|r\|_\infty}{\sqrt{h}} |E^k|_A^2.
\ee
However, because of the factor $\frac{1}{\sqrt{h}}$, this bound is not sufficient in order to conclude as in the previous proof.
\end{rem}
}\fi
%
%

\bigskip

\subsection{A universal stability lemma}\label{sec:stab-lemma}
In the following, it is assumed that $|\cdot|$ is any vectorial norm. We will use the result for the canonical Euclidean norm $|\cdot|\equiv \|\cdot\|$
and the $A$-norm $|\cdot|\equiv |\cdot|_A$.

In order to prove the following Lemma \ref{lem:stability}, we will exploit properties of the matrix 
\be
  M_\tau:=
  \begin{pmatrix} 
    (3-C\tau)  & -4  & 1     &  0      &  \\
    0  &  \phantom{-}(3-C\tau)  & -4    &  \ddots & \ddots  \\
       &  \phantom{-}0  & \ddots&\ddots&  1   \\
       &     & \ddots&\ddots&   -4  \\
       &     &       &  0    &  \phantom{-}(3-C\tau)   
  \end{pmatrix},
  \label{eq:Mtau}
\ee
in particular the fact that $(M_\tau)^{-1}\geq 0$ for $\tau$ small enough (which we prove).

\begin{lemma}\label{lem:stability}
Assume that there exists a constant $C\geq 0$ such that $\forall k=2,\dots,N$:
\be
  && \hspace{-1cm} \frac{1}{2}\Big( (3-C\tau)|E^k|^2 -4| E^{k-1}|^2 +|E^{k-2}|^2\Big)+|E^k-E^{k-1}|^2-|E^{k-1}-E^{k-2}|^2
     \nonumber \\
  & & \hspace{8.3cm}
  \leq 2 \tau |E^k|\;|\Eps^k|.
  \label{eq:one_step}
\ee
Then there exists a constant $C_1\geq 0$ and $\tau_0>0$ such that $\forall 0<\tau \leq \tau_0$, 
$\forall n\le N$:
\be \label{eq:stabnormL2}
  \max_{2\leq k\leq n}|E^k|^2 
   & \leq & C_1 \Big(|E^0|^2 + |E^1|^2 + \tau \sum_{2\leq j\leq n}|\Eps^j|^2\Big).
\ee
\end{lemma} 
\begin{proof}
Let us denote
$$
x_k:= |E^k|^2\qquad \text{and}\qquad y_k:=|E^{k}-E^{k-1}|^2,
$$
so that \eqref{eq:one_step} reads
\be\label{eq:one_stepxy}
\Big( (3-C\tau) x_k-4x_{k-1} +x_{k-2}\Big)\leq  2(y_{k-1} -y_k) + 4 \tau |E^k|\;|\Eps^k|.
\ee
For a given $\tau>0$ and given $k$, let $M_\tau \in \R^{(k-1)\times (k-1)}$ as defined in \eqref{eq:Mtau}.
Let $z, w \in \R^{k-1}$ be defined by 
$$
z:=(x_k, x_{k-1}, \dots, x_2)^T \quad\text{and}\qquad 
w:=(2(y_{j-1}-y_j) + 4\tau |E^j|\;|\Eps^j|)_{j=k,\dots,2}.
$$
By \eqref{eq:one_stepxy}, we have
\be\label{eq:mozw} 
   M_\tau z \leq w.
\ee
We notice that $M_\tau=(3-C\tau) I -4 J + J^2$ with
$$
  J:=\tridiag(0,0,1).
$$
Hence, with
$$
  \lambda_{1}= 2 +  \sqrt{1+C\tau} \quad \mbox{and} \quad
  \lambda_{2}= 2 -  \sqrt{1+C\tau},
$$
the roots of $\ml^2 - 4\ml + (3-C\tau) =0$ for $3- C\tau\geq0$,
we can write
$$
  M_\tau= (\lambda_1 I-J)(\lambda_2I-J) = \lambda_1\lambda_2\left(I-\frac{J}{\lambda_1}\right)\left(I-\frac{J}{\lambda_2}\right).
$$
Furthermore, since $J^{k-1}=0$, it holds
\begin{align*}
  M_\tau^{-1} & =   \frac{1}{\lambda_1\lambda_2}\left(I-\frac{J}{\lambda_1}\right)^{-1}\left(I-\frac{J}{\lambda_2}\right)^{-1}\\
  & =  \frac{1}{\lambda_1\lambda_2} \left(\sum_{0\leq q \leq k-2} 
    \left(\frac{J}{\lambda_1}\right)^q\right) \left(\sum_{0\leq q \leq k-2} \left(\frac{J}{\lambda_2}\right)^q\right)
  = \sum_{p=0}^{k-2} a_p J^p,
\end{align*}  
where 
$$
  a_p:=\sum^p_{j=0} \frac{1}{\lambda_1^{j+1}\lambda_2^{p-j+1}}
  = \frac{1}{\lambda_2^{p+2}}\sum^p_{j=0} \left(\frac{\lambda_2}{\lambda_1}\right)^{j+1}.
$$
Therefore $M_\tau^{-1}\geq 0 $ componentwise (for $\tau<3/C$), and using \eqref{eq:mozw} it holds $z \leq M_\tau^{-1} w$. 

It is possible to prove that there exists $\tau_0>0$ 
and a constant $C_0\geq 0$ (depending only on $T$)
such that  $\forall 0<\tau\leq \tau_0$ and $\forall p\geq 0$:
\be\label{eq:a}
  0\leq a_p\leq C_0\quad\text{and}\quad a_p-a_{p-1}\geq 0.
\ee
We postpone the proof of \eqref{eq:a} to the end.
For the first component of $z$, we deduce
\be
   x_k
  & \leq & \sum_{j=0}^{k-2} a_j w_{j+1} \nonumber \\
  & \leq & 2  \sum_{j=0}^{k-2} a_j (y_{k-j-1}-y_{k-j})
        + {4C_0\tau} \sum_{j=2}^k |E^j|\; |\Eps^j|  \label{eq:ineq1} \\
  & = &  -2a_0 y_k + 
  2  \sum_{j=0}^{k-3} (a_j-a_{j+1}) y_{k-j+1}
  + 2 a_{k-2} y_1 + {4C_0\tau} \sum_{j=2}^k |E^j|\; |\Eps^j|,\nonumber
\ee
for all $k\geq 2$,
where, for \eqref{eq:ineq1}, we have used the fact that $a_p\leq C_0$.
Since $y_j\geq 0$, $\forall j$, by definition, $a_{k-2}\leq C_0$, $a_0= \frac{1}{\lambda_1\lambda_2}\geq 0$ and $a_j-a_{j-1}\geq 0$, $\forall j$, 
we obtain
\be\label{eq:inter-estim}
   x_k & \leq & 2 C_0 y_1 + {4C_0\tau} \sum_{j=2}^k |E^j|\; |\Eps^j|.
\ee
Recalling the definition of $x_k$ and $y_k$, for any $2\leq k\leq n$ one has:
\beno
  |E^k|^2 
  &\leq & 2C_0 |E^1-E^0|^2 +  4C_0\tau \sum_{j=2}^k |E^j|\; |\Eps^j|\\
  &\leq & 4C_0 (|E^0|^2 + |E^1|^2) + 4C_0\tau\Big(\max_{2\leq k\leq n}|E^k|\Big)\sum_{j=2}^n |\Eps^j|\\
  &\leq & 4C_0 (|E^0|^2 + |E^1|^2) + \frac{1}{2}\Big(\max_{2\leq k\leq n}|E^k|\Big)^2 + {8C_0^2\tau^2} \Big(\sum_{j=2}^n |\Eps^j|\Big)^2
\eeno
(where we made use of $2ab\leq \frac{a^2}{K} +K b^2 $ for any $a,b\geq0$ and $K>0$).
Hence, we obtain
\beno
 \Big(\max_{2\leq k\leq n}|E^k|\Big)^2 
   & \leq & C_1\Big(|E^0|^2 +  |E^1|^2 +  \tau \sum_{j=2}^n |\Eps^j|^2\Big)
\eeno
with $C_1 := \max(8C_0, 16 C_0^2 T)$ 
(we used $\Big(\sum_{j=2}^n |\Eps^j|\Big)^2 \leq n \sum_{j=2}^n |\Eps^j|^2$ and $n\tau \leq T$).

It remains to prove \eqref{eq:a}.
From the definition of $a_p$ one has
$$
  a_p  = \frac{1}{\lambda_2^{p+2}}\sum^p_{j=0} \left(\frac{\lambda_2}{\lambda_1}\right)^{j+1}  
    \leq   \frac{1}{\lambda_2^{p+2}} \left(1-\frac{\lambda_2}{\lambda_1}\right)^{-1}
$$
for $p=0,\ldots, k-2$.
Observing that $\frac{\lambda_2}{\lambda_1}\leq \frac{1}{3}$, it follows that
$$
  a_p \leq  \frac{3}{2 \lambda_2^{p+2}} \leq \frac{3}{2 (2-\sqrt{1+C\tau})^{n}}.
$$
Notice that $\sqrt{1+C\tau}\leq 1+C\tau$, and also that $e^{-x} \leq 1- x/2$, $\forall x\in[0,1]$.
Hence
$(2-\sqrt{1+C\tau})^n\geq (2-(1+C\tau))^n = (1-C\tau)^n \geq (e^{-2C\tau})^n = e^{-2 C t_n}$ for $C\tau \leq \frac{1}{2}$,
and therefore $a_p\leq \frac{3}{2} e^{2Ct_n}$. The desired result follows with $C_0:=\frac{3}{2} e^{2CT}$ and $\tau_0:=\frac{1}{2C}$.

Moreover, one has
$$
  a_p-a_{p-1} 
    = \frac{1}{\lambda_2^{p+1}}\left(\frac{1}{\lambda_2}\sum^p_{j=0} 
    \left(\frac{\lambda_2}{\lambda_1}\right)^{j+1}- \sum^{p-1}_{j=0} \left(\frac{\lambda_2}{\lambda_1}\right)^{j+1} \right),
$$
which is nonnegative for $\tau$ small enough thanks to the fact that $\lambda_1,\lambda_2\geq 0$ and $\lambda_2\leq 1$. 
\end{proof}






\COMMENT{Observe that there are two main differences: 
first Lemma \ref{lem:nonlin2lin} is not necessary anymore since 
it is sufficient to consider the optimal control $\bar a^k_i$ for $H[u^k]_i$ to get a linear recursion. 
Second the consistency error is replaced by $\ell$. 
Moreover, boundary term have to be taken into account in this case. \cblue{\fbox{To be completed.}}  
\vio{
\fbox{AP:} When studying the boundedness we would have (see above) 
\be\label{eq:stab_u}
  \frac{1}{2}\Big( (3-C\tau)|u^k|_A^2 -4|u^{k-1}|_A^2 +|u^{k-2}|_A^2\Big)+|u^k-u^{k-1}|_A^2-|u^{k-1}-u^{k-2}|_A^2\leq 2 \tau |u^k|_A\;|\ell^k|_A+2\tau\frac{(u^k_1-u^k_0)^2}{h^2} .
\ee
Adding this term in the proof of Lemma \ref{lem:stability} it seems to me that we get the following:
\beno
  \max_{2\leq k\leq n}|u^k|^2_A 
   & \leq & C_1 \Big(|u^0|^2_A + |u^1|^2_A + \tau \sum_{2\leq j\leq n}|\ell^j|_A^2 +\tau \sum_{2\leq j\leq n} \frac{(u^j_1-u^j_0)^2}{h^2}\Big).
\eeno
P.S. With variable sign of the drift we have also $\tau \sum_{2\leq j\leq n} \frac{(u^j_{I+1}-u^j_{I})^2}{h^2}$
}
}

\section{Stability in the Euclidean norm} \label{sec:L2}

The fundamental stability result given by  Lemma \ref{lem:stability} 
applies to any vectorial norm.
In this section, we discuss some special cases where \eqref{eq:one_step} can be obtained for
the Euclidean norm $|\cdot|=\|\cdot\|$.

We first prove the stability result for this norm under the extra assumption (A3), i.e., the control may appear 
except in the diffusion term, which must also be Lipschitz continuous in the following proof.

\subsection{Proof of Theorem~\ref{th:L2stab} {\bf (stability in the Euclidean norm)} }\label{sec:l2} 


We consider the scalar product of \eqref{eq:err-recu-lin_gen} directly with $E^k$ (instead of $A E^k$ previously used),
again in the situation where $b\geq 0$ to simplify the argument. We obtain:
\begin{align}\label{eq:scalarL2}
  \langle E^k,3E^k - 4 E^{k-1}+E^{k-2}\rangle 
   + 2\tau \langle E^k,\Delta^k A E^k + F^k B E^k + R^k E^k\rangle = -2\tau\langle E^k,\Eps^k\rangle.
\end{align}
As in Section~\ref{sec:vect2scalar}, we have
\be
  & & \langle E^k,3E^k - 4 E^{k-1}+E^{k-2}\rangle \label{eq:timeL2} 
    \label{eq:newscalar}\\
  & & \hspace{1cm} \geq \frac{1}{2} \left( 3 \|E^k\|^2  - 4\|E^{k-1}\|^2 +\|E^{k-2}\|^2\right) + \|E^k-E^{k-1}\|^2 - \|E^{k-1}-E^{k-2}\|^2.
    \nonumber
\ee
We now focus on bounding the other terms on the left-hand side of \eqref{eq:scalarL2}.

By using the Lipschitz continuity of $\sigma^2$ 
one has 
\begin{align}
 \langle E^k,\Delta^k A E^k\rangle 
  & = \sum_{i\in\mathbb I} \frac{(\sigma^k_i)^2}{2h^2}(-E^k_{i+1}+2E^k_{i}-E^k_{i-1}) E^k_i\nonumber\\
  & = \sum_{i\in\mathbb I}\frac{(\sigma^k_{i-1})^2}{2h^2}(E^k_{i-1}-E^k_{i})^2
     +\sum_{i\in\mathbb I} 
       \left(\frac{(\sigma^k_{i-1})^2}{2h^2}-\frac{(\sigma^k_i)^2}{2h^2}\right) (E^k_{i-1}-E^k_i) E^k_i\nonumber\\
  & \geq \frac{\eta}{2h^2} \sum_{i\in\mathbb I}(E^k_{i-1}-E^k_{i})^2 
    - \frac{L}{2h} \sum_{i\in\mathbb I}|E^k_{i-1}-E^k_{i}| |E^k_i|. \nonumber
\end{align}
Therefore, by the Cauchy-Schwarz inequality, one obtains
\be\label{eq:L2est-diff}
\langle E^k,\Delta_k A E^k\rangle \geq  \frac{\eta}{2h^2}\|\delta E^k\|^2 - \frac{L}{2h} \|\delta E^k\|\|E^k\|,
\ee
where $\delta E^k$ is defined by \eqref{eqdef:deltaE}.
Moreover, for the first order term one has 
\begin{align}
  \langle E^k,F^k B E^k\rangle  
  & =  \sum_{i\in\mathbb I} \frac{b_i}{2h}(3E^k_{i}-4E^k_{i-1}+E^k_{i-2}) E^k_i\nonumber\\
  & \geq -  \frac{3 \|b\|_\infty}{2h}\sum_{i\in\mathbb I}|E^k_{i}-E^k_{i-1}| |E^k_i| - \frac{\|b\|_\infty}{2h} \sum_{i\in\mathbb I}|E^k_{i-1}-E^k_{i-2}| |E^k_i|\nonumber\\
  & \geq  -  \frac{2 \|b\|_\infty}{h}\|\delta E^k\|\|E^k\|,\label{eq:L2est-drift}
\end{align}
where for the last equality we have used that $\|\delta^2 E^k\|\leq \|\delta E^k\|$.
Putting together estimates 
\eqref{eq:L2est-diff} and \eqref{eq:L2est-drift}, using the fact that $\langle E^k,R^k E^k\rangle\geq -\|r\|_\infty \|E^k\|^2$,
we get
\begin{align*}
  \langle E^k,\Delta_k A E^k + F_k B E^k + R^k E^k\rangle 
    & \geq \frac{\eta}{2h^2} \|\delta E^k\|^2 - \frac{C_1}{2h} \|\delta E^k\|\|E^k\|-\|r\|_\infty\|E^k\|^2\\
    & \geq \frac{\eta}{4h^2} \|\delta E^k\|^2 - \left(\frac{C_1^2}{4\eta} + \|r\|_\infty\right) \|E^k\|^2,
\end{align*}
where we have denoted $C_1:=L+4\|b\|_\infty$ and have used again the Cauchy-Schwarz inequality.
Hence, together with \eqref{eq:timeL2}, this gives \eqref{eq:one_step} with $|\cdot|=\|\cdot\|$ and the constant
$C:=4 (\frac{C_1^2}{4\eta}+\|r\|_\infty)$.
By using Lemma~\ref{lem:stability}, this concludes the proof of Theorem~\ref{th:L2stab}. 
\hfill$\Box$



\subsection{Linear equation with degenerate diffusion term} \label{sec:linear1d}
The next result concerns the case of a possibly degenerate diffusion term. 
It will require more restrictive assumptions on the drift and diffusion terms, and we shall assume that there is no control here.
Indeed, in this case, one cannot count on the positive term coming from the non-degenerate diffusion which, 
in the proof of 
Theorem~\ref{th:L2stab},
is used to compensate the negative correction terms coming from the drift term. 
This leads us to consider the following assumptions:

\medskip
{\sc Assumption (A4).} $r$ is bounded. The drift and diffusion coefficients are independent of the control, i.e.\
$b \equiv b(t,x)$ and $\ms \equiv \ms(t,x)$,
 and there exist $L_1, L_2\geq 0$ such that,
for all $t,x,h$:
\be
 & & |b(t,x)-b(t,y)| \leq L_1 |x-y|, \\ 	 
 & &
 \frac{\ms^2(t,x-h)- 2\ms^2(t,x) + \ms^2(t,x+h)}{h^2}  \geq -L_2.\label{eq:semiconc}    
\ee

(The last condition is equivalent to $(\ms^2)_{xx}\geq -L_2$ in the differentiable case.)

\begin{prop}\label{prop:L2_1d_deg}
Let assumption (A4) be satisfied. Then \eqref{eq:one_step} holds for $|\cdot|=\|\cdot\|$.
\end{prop}

\begin{proof}
We consider again the scalar recursion \eqref{eq:scalarL2}. 
%
%
%
%
For any vector $E=(E_i)_{1\leq i\leq I}$ (with $E_j=0$ for $j\in\{-1,0,I+1,I+2\}$),
it holds:
\beno
  E_i(2E_i - E_{i-1} - E_{i+1})
    & \geq & 2 |E_i|^2  - \fud (|E_i|^2+|E_{i-1}|^2) - \fud (|E_i|^2+|E_{i+1}|^2) \\
    & \geq & \fud \ (2 |E_i|^2 - |E_{i-1}|^2 - |E_{i+1}|^2).
\eeno
Hence, by the semi-concavity assumption \eqref{eq:semiconc} on $\ms^2$,
\be
   \langle E^k,\Delta^k A E^k\rangle
  &  =   &  \sum_{1\leq i\leq I} \frac{\ms_i^2}{2h^2} E^k_i (2E^k_i - E^k_{i-1} - E^k_{i+1})  \nonumber\\
  & \geq &  \sum_{1\leq i\leq I} \frac{\ms_i^2}{4h^2} (-|E^k_{i-1}|^2 + 2  |E^k_{i}|^2 - |E^k_{i+1}|^2) \nonumber\\
  & \geq &  \sum_{1\leq i\leq I}\bigg(\frac{ - \ms_{i-1}^2 + 2\ms_i^2 - \ms_{i+1}^2}{4h^2}\bigg)\,|E^k_i|^2. \nonumber\\
  & \geq & -  \frac{L_2}{4} \|E^k\|^2. \label{eq:DAbound} 
\ee
Now we focus on a lower bound for  $\langle E^k,F^k B E^k\rangle$. 
Let $y^k_i=|E^k_i-E^k_{i-1}|^2$.
First, 
\beno
  (3E^k_{i}-4E^k_{i-1}+E^k_{i-2}) E^k_i 
   &  = &  \fud (3 |E^k_{i}|^2- 4 |E^k_{i-1}|^2 + |E^k_{i-2}|^2)  \\
   && \hspace{2.1 cm} + \fud (4 |E^k_i - E^k_{i-1}|^2 - |E^k_i - E^k_{i-2}|^2) \\
   & \geq &  \fud (3 |E^k_{i}|^2- 4 |E^k_{i-1}|^2 + |E^k_{i-2}|^2)  + \fud (2 y^k_i - 2 y^k_{i-1}). 
\eeno

We assume again $b_i\geq 0$ for all $i$ to simplify the presentation. 
The case where 
$b_i\leq 0$ for some $i$ is similar. 
Then, the following bound holds:
\beno
   \< E^k,\, F^k B E^k \>
   & =    & \sum_{i=1}^I  \frac{b_i}{2h}(3E^k_{i}-4 E^k_{i-1}+E^k_{i-2}) E^k_i =  \sum_{i=1}^{I+2} \frac{b_i}{2h}(3E^k_{i}-4 E^k_{i-1}+E^k_{i-2}) E^k_i \\
   & \geq & \sum_{i=1}^{I+2}  \frac{b_i}{4h} (3 |E^k_{i}|^2- 4 |E^k_{i-1}|^2 + |E^k_{i-2}|^2)  
          + \sum_{i=1}^{I+2}  \frac{b_i}{h} (y^k_i-y^k_{i-1})  \\
   & \geq  & \sum_{i=1}^I  \bigg(\frac{3 b_i - 4 b_{i+1} + b_{i+2}}{4h}\bigg) |E^k_{i}|^2
          + \sum_{i=1}^{I+1}  \bigg(\frac{b_i-b_{i+1}}{h}\bigg) y^k_i 
\eeno
(where we have used $y^k_{0}=y^k_{I+2}=0$
and $\sum_{1\leq i\leq I+2} b_i (E^k_{i-2})^2 = \sum_{1\leq i\leq I} b_{i+2} (E^k_i)^2 $
as well as
$\sum_{1\leq i\leq I+2} b_i (E^k_{i-1})^2 = \sum_{0\leq i\leq I+1} b_{i+1} (E^k_i)^2 
 = \sum_{1\leq i\leq I} b_{i+1} (E^k_i)^2$).
Then, by the Lipschitz continuity of $b(.)$ and the bound 
$y^k_i\leq 2(E^k_i)^2 + 2(E^k_{i-1})^2$, we have
\be
   \<  E^k,\, F^k B E^k  \> 
      & \geq &   - L_1 \sum_{i=1}^I |E^k_i|^2 - L_1 \sum_{i=1}^{I+1} y^k_i \geq - 3 L_1 \|E^k\|^2.
      \label{eq:FBbound}
\ee
By combining the bounds \eqref{eq:DAbound} and \eqref{eq:FBbound}, we obtain
$$
   \< E^k,\, \Delta^k A E^k \> + \<  E^k,\, F^k B E^k  \>  + \< E^k,\, R^k E^k\>
    \geq   - (\frac{L_2}{4}  + 3 L_1 + \|r\|_\infty) \|E^k\|^2.
$$
Therefore, inequality \eqref{eq:one_step} is obtained with $C := 4(\frac{L_2}{4} + 3 L_1 + \|r\|_\infty)$,
which leads to the desired stability estimate.
\end{proof}

\if{

Moreover, one has 
\begin{align*}
\langle E^k,F^k B E^k\rangle  & =  \sum_{i\in\mathbb I} \frac{b^k_i}{2h}(3e^k_{i}-4e^k_{i-1}+e^k_{i-2}) e^k_i \\
 & \hspace{-1.3 cm} =   \sum_{i\in\mathbb I} \frac{4 b^k_i}{2h}(e^k_{i}-e^k_{i-1})e^k_i - \frac{b^k_i}{2h}(e^k_{i}-e^k_{i-2}) e^k_i   \\
& \hspace{-1.3 cm} = \sum_{i\in\mathbb I} \frac{b^k_i}{h}\Big((e^k_{i})^2+(e^k_i-e^k_{i-1})^2 -(e^k_{i-1})^2\Big) - \frac{b^k_i}{4h}\Big((e^k_{i})^2+(e^k_i-e^k_{i-2})^2 -(e^k_{i-2})^2\Big)\\
& \hspace{-1.3 cm} \geq  \sum_{i\in\mathbb I}\frac{b^k_i}{4h}\Big(3(e^k_{i})^2 -4(e^k_{i-1})^2 +(e^k_{i-2})^2 \Big) \pm  \frac{b^k_{i-1}}{h}(e^k_{i-1})^2 \pm \frac{b^k_{i-2}}{4h}(e^k_{i-2})^2\\
& \hspace{-1.3 cm} \quad+ \frac{b^k_i}{2h}\Big((e^k_i-e^k_{i-1})^2-(e^k_{i-1}-e^k_{i-2})^2\Big) \pm \frac{b^k_{i-1}}{2h}(e^k_{i-1}-e^k_{i-2})^2.
\end{align*}

Observing that (using the asymptotic zero boundary conditions)
\begin{eqnarray*}
\sum_{i\in\mathbb I}b^k_{i-j}(e^k_{i-j})^2 =\sum_{i\in\mathbb I}b^k_{i}(e^k_{i})^2, \, j=1,2, \quad
\sum_{i\in\mathbb I}b^k_{i-1}(e^k_{i-1}-e^k_{i-2})^2 = \sum_{i\in\mathbb I}b^k_{i}(e^k_{i}-e^k_{i-1})^2,
\end{eqnarray*}
one obtains
\begin{align*}
\langle E^k,F^k B E^k\rangle  & \geq  \sum_{i\in\mathbb I} \Big(\frac{b^k_{i-1}}{h}-\frac{b^k_{i}}{h}\Big)(e^k_{i-1})^2 +\Big(\frac{b^k_{i}}{4h}-\frac{b^k_{i-2}}{4h}\Big)(e^k_{i-2})^2 \\
& \quad +\Big(\frac{b^k_{i-1}}{2h}-\frac{b^k_{i}}{2h}\Big)(e^k_{i-1}-e^k_{i-2})^2 \\
 &  \geq -L\sum_{i\in\mathbb I}(e^k_{i-1})^2 -\frac{L}{2} \sum_{i\in\mathbb I}(e^k_{i-2})^2 -L\sum_{i\in\mathbb I}(e^k_{i-1})^2 -L \sum_{i\in\mathbb I}(e^k_{i-2})^2\\ 
 & \geq  -\frac{7 L}{2} \|E^k\|^2.
\end{align*} 
Therefore, inequality \eqref{eq:one_step} is obtained with $C = 7L+2\|r\|_\infty + {L_1}$.

Applying now Lemma \ref{lem:stability} we obtain the following result:

\begin{theorem}
Let either assumption (A1),(A2),(A3) or (A4) be satisfied, as well as the CFL condition \eqref{eq:CFL}. Then 
there exists a constant $C\geq 0$ (independent of $\tau$ and $h$) and some $\tau_0>0$ such that, for any $\tau\leq \tau_0$
\be\label{eq:errorL2}
\max_{2\leq k\leq N}\|E^k\|^2 
   & \leq & C \Big(\|E^0\|^2 + \|E^1\|^2 +\tau\sum_{2\leq k\leq N} \|\Eps^k\|^2\Big).
\ee
\end{theorem}
}\fi

\subsection{Extension to a two-dimensional case}

Under suitable assumptions, the  result of Theorem \ref{th:L2stab} can be extended to multi-dimensional equations. The nonlinearity can be treated exactly as in Section \ref{sec:vec2lin} (or \ref{sec:3.4}), so that we can focus on the linear case
{
$$
  v_t -\frac{1}{2}\trace[\Sigma(t,x)D_x^2 v] + b(t,x)D_x v+r(t,x) v +\ell(t,x)=0
$$
}
for a positive definite matrix $\Sigma$ and a drift vector $b$.
For simplicity, we furthermore consider the two-dimensional case $d=2$, with $r,\ell\equiv 0$, 
and omit the dependence of the coefficients {on the time variable}, then with 
$$
  \Sigma(x,y) :=\left(\begin{array}{cc}
  \sigma^2_1(x,y) & \rho \sigma_1\sigma_2(x,y)\\
  \rho \sigma_1\sigma_2(x,y) &  \sigma_2^2(x,y)
  \end{array}\right)
 \quad\text{and}\quad 
  b(x,y): =\left(\begin{array}{c}
  b_1(x,y)\\
  b_2(x,y)
\end{array}\right),
$$ 
where  $\ms_1,\ms_2\geq 0$ and $\rho\in[-1,1]$ is the correlation parameter, the equation reads
$$
  v_t - \frac{1}{2}\sigma_1^2(x,y)v_{xx} -\rho \sigma_1\sigma_2(x,y)v_{xy}-\frac{1}{2}\sigma_2^2(x,y)v_{yy} + b_1(x,y)v_x + b_2(x,y)v_y = 0.
$$
The computational domain is given by $\Omega:=(x_{\min},x_{\max})\times (y_{\min},y_{\max})$.
We introduce the discretization in space defined by the steps $h_x, h_y>0$ and we denote by $\mathcal G_{(h_x,h_y)}$ the associated mesh. 
In what follows, given any function $\phi$ of $(x,y)\in \Omega$, we will denote $\phi_{ij}=\phi(x_i,y_j)$ for {$(i,j)\in\mathbb I:=\mathbb I_1\times\mathbb I_2$, where  $\mathbb I_1=\{1, \ldots, I_1\}$, $\mathbb I_2=\{1, \ldots, I_2\}$}.

Assuming that $\rho\ge 0$ everywhere (the case when $\rho\leq 0$ is similar),
we consider a 7-point stencil for the second order derivatives (see \cite[Section 5.1.4]{HackbuschBook}):
\begin{align*}
 v_{xx}\sim \frac{v_{i-1,j}-2v_{ij}+v_{i+1,j}}{h_x^2}=:D^2_{xx}v_{ij},\quad  \quad v_{yy}\sim \frac{v_{i,j-1}-2v_{ij}+v_{i,j+1}}{h_y^2}=:D^2_{yy}v_{ij}\\
 v_{xy} \sim \frac{-v_{i,j-1}-v_{i,j+1}-v_{i-1,j}-v_{i+1,j}+v_{i-1,j-1}+v_{i+1,j+1}+2v_{ij}}{2 h_x h_y}=:D^2_{xy}v_{ij}
 \end{align*}
 and the BDF approximation of the first order derivatives
\begin{align*}
D^{1,-}_x u_{ij}:= \frac{3 u_{ij} - 4 u_{i-1,j} + u_{i-2,j}}{2 h_x}
   \quad \mbox{and} \quad 
   D^{1,+}_x u_{ij}:= -\bigg(\frac{3 u_{ij} - 4 u_{i+1,j} + u_{i+2,j}}{2h_x}\bigg),\\
D^{1,-}_y u_{ij}:= \frac{3 u_{ij} - 4 u_{i,j-1} + u_{i,j-2}}{2 h_y}
   \quad \mbox{and} \quad 
   D^{1,+}_y u_{ij}:= -\bigg(\frac{3 u_{ij} - 4 u_{i,j+1} + u_{i,j+2}}{2h_y}\bigg).
 \end{align*}
The scheme is therefore defined, for $k\geq 2$, by
\be
  \label{eq:scheme2d}
  & & 0 \ =\ \frac{u^k_{ij}-4u^{k-1}_{ij}+u^{k-2}_{ij}}{2\tau} 
  \\ 
  & & \hspace{0.5cm} 
       -  \frac{1}{2} \ms_1^2    (x_i,y_j)D^2_{xx}u^k_{ij} 
       -\rho          \ms_1\ms_2 (x_i,y_j)D^2_{xy}u^k_{ij}
       -  \frac{1}{2} \ms_2^2    (x_i,y_j)D^2_{yy}u^k_{ij}\nonumber\\
  & & \hspace{0.5cm}
        + b^+_1(x_i,y_j) D^{1,-}_x u^{k}_{ij} -  b_1^-(x_i,y_j) D^{1,+}_x u^{k}_{ij}  
	+ b^+_2(x_i,y_j) D^{1,-}_y u^{k}_{ij} -  b_2^-(x_i,y_j) D^{1,+}_y u^{k}_{ij}.
  \nonumber
\ee
A straightforward calculation shows that the second order term also reads
\begin{align}
\nonumber
& \sigma^2_1(x_i,y_j)D^2_{xx}u_{ij} +2\rho \sigma_1\sigma_2(x_i,y_j) D^2_{xy}u_{ij}+\sigma^2_2(x_i,y_j)D^2_{yy}u_{ij}\\
&\hspace{2.5cm} =
   \alpha_{ij} D^2_{xx}u_{ij}  
 +  \beta_{ij} D^2_{yy}u_{ij}   
 + \gamma_{ij}\left(u_{i-1,j-1}-2u_{ij}+u_{i+1,j+1}\right),
 \label{abc}
\end{align}
with 
\begin{align*}
  \alpha_{ij} & :=\frac{\sigma_1(x_i,y_j)}{h_x}\left(\frac{\sigma_1(x_i,y_j)}{h_x}-\frac{\rho\sigma_2(x_i,y_j)}{h_y}\right), \\
  \beta_{ij}  & :=\frac{\sigma_2(x_i,y_j)}{h_y}\left(\frac{\sigma_2(x_i,y_j)}{h_y}-\frac{\rho\sigma_1(x_i,y_j)}{h_x}\right), \qquad 
  \gamma_{ij} :=\frac{\rho(x_i,y_j)\sigma_1(x_i,y_j)\sigma_2(x_i,y_j)}{h_y h_x}.
\end{align*}
{The scheme is completed with the following boundary conditions: 
\begin{align*}
u^k_{i,j} = v(t_k,x_i, y_j),& \quad \forall i\in\{-1,0\}\cup \{I_1+1,I_1+2\}, \; j\in\mathbb I_2, \\
u^k_{i,j} = v(t_k,x_i, y_j),& \quad \forall j\in\{-1,0\}\cup \{I_2+1,I_2+2\}, \; i\in\mathbb I_1.
\end{align*}
}

For simplicity, assume $h_x=h_y=: h$.
We consider the following assumptions:

\medskip
\noindent
{\sc Assumptions} \\
{\sc (A1'):}
  $\|b_i\|_\infty<\infty$ for $i=1,2$;\\
{\sc (A2'):}
$
  \mbox{$\exists \eta>0$, $\forall (x,y)\in\mO$, $\forall i\neq j$:
     $\sigma_i^2(x,y)-\rho(x,y)\sigma_i(x,y)\sigma_j(x,y)\geq \eta$};
$ \\
{\sc (A3'):}
$\forall i,j=1,2$, $\ms_i\ms_j$ is Lipschitz continuous on $\mO$.


\medskip

We then have the following result.
The proof is similar to the one of Theorem~\ref{th:L2stab}, using
(\ref{abc}) with $\alpha_{ij}, \beta_{ij}, \gamma_{ij} \ge 0$ by assumption (A2'),
 {and is therefore omitted}.

\begin{prop}\label{prop:L2_2d}
Let assumptions (A1'),(A2') and (A3') be satisfied. Then the stability estimate \eqref{eq:stabnormL2} holds for $|\cdot|=\|\cdot\|$.
\end{prop}


\begin{rem}
$(i)$ 
If $h_x\neq h_y$ and for instance $h_y=C h_x$ for some $C\geq 1$,
{(A2') has to hold with $\sigma_2$ replaced by $\sigma_2/C$ as a result of the scaling properties of the scheme.}

$(ii)$ 
Observe that  assumption (A2') is equivalent to requiring strong diagonal dominance of the covariance matrix.

$(iii)$ 
When the strong diagonal dominance of the matrix $\Sigma$ is not guaranteed, one can consider the generalized finite difference scheme in \cite{BZ03}. {
However, determining the precise set of assumptions on the coefficients needed to apply the previous arguments does not seem easy from the construction in \cite{BZ03}.}

\if{
Given a set of directions $\mathcal S^k_{ij} \subset \mathbb{N}\times \mathbb{N}$ for each triplet $(t_k,x_i,y_j)$,
an approximation of the second order term is defined by 
$$
\frac{1}{2}\trace[\Sigma(t_k,x_i,y_j)D^2v]\sim\sum_{\xi\in\cS^k_{i,j}} \alpha^\xi_{i,j} D^2_\xi v_{ij},
$$
where
$$
D^2_\xi u_{ij} := u_{i+\xi_1,j+\xi_2} -2 u_{ij} + u_{i-\xi_1,j-\xi_2}
$$
for given $\xi\equiv(\xi_1,\xi_2)$ and $\alpha^\xi_{i,j}\ge 0$ are chosen to make the scheme consistent.
In the case of the seven-point stencil studied above, $\cS^k_{ij}=\{(1,0),(0,1),(1,1)\}$ $\forall (i,j)\in\mathbb I,k\geq 1$  and $\alpha^{(1,0)}=\alpha$, $\alpha^{(0,1)}=\beta, \alpha^{(1,1)}= \gamma$. In order to apply directly the same argument as above, we would need all coefficients $\alpha^\xi$ to be Lipschitz continuous and non-degenerate (to compensate the drift term). This requirement does not seem easy to verify generically from the construction in \cite{BZ03}.

{We also point out that  same techniques can be used to  extended the results to higher dimension under some suitable choice  of the finite difference approximation of the second order term.}
}\fi
\end{rem}

\section{Error estimates}\label{sec:error}

In this section, we give detailed error estimates for the implicit BDF2 scheme \eqref{eq:BDF2}.
We consider the following rescaled norms on $\R^I$: 
\begin{eqnarray*}
|u|_0 := \left(\sum_{i\in \mathbb{I}} u_i^2 \, h \right)^{1/2} = \, \|u\| \sqrt{h}, \qquad
|u|_1 := \left(\sum_{i\in \mathbb{I}} \left(\frac{u_i-u_{i-1}}{h} \right)^2 h \right)^{1/2} = \, |u|_A \sqrt{h},
\end{eqnarray*}
corresponding to discrete approximations of $L^2(\mO)$- and $H^1(\mO)$ norms, respectively.
{
Both these norms
will be used in the forthcoming numerical section.}
\if{
{\sc Assumption (A5).}
The function $v$ belongs  to $C^{1+\delta, 2+\delta}$ for some $\delta \in (0,1]$ in the following sense:
$v$ is in $C^{1,2}([0,T]\times \bar\mO,\R)$ and $\exists L>0\ \forall x,y\in\Omega$, $\forall s,t\in [0,T]$
\begin{eqnarray*}
  |v_t (t,x) - v_t (s,x)| + |(D^2_x v) (t,x) - (D^2_x v) (t,y)| 
  \leq L \left( |t-s|^\delta  +  |x-y|^\delta \right).
\end{eqnarray*}
}\fi

In addition, we define the following semi-norm on some interval $\mathcal{I}=(a,b)$:
$$ 
  |w|_{C^{0,\ma}(\mathcal{I})}:= \sup\bigg\{ \frac{|w(x) - w(y)|}{|x-y|^\ma}, \ x\neq y, \ x,y\in \mathcal{I} \bigg\}.
$$

For a given open subset $\mO_T^*$ of $(0,T)\times \Omega$, 
we define $C^{k,\ell}(\mO_T^*)$ 
as the set of functions $v:\mO_T^*\conv \R$ which admit continuous derivatives
$(\frac{\partial^{i}v}{\partial t^i})_{0\leq i\leq k}$ and 
$(\frac{\partial^{j}v}{\partial x^j})_{0\leq j\leq \ell}$ 
on $\mO_T^*$.
We also denote by $C^{k,\ell}_b(\mO_T^*)$ the subset of functions with bounded derivatives on $\mO_T^*$.

\medskip 
{\sc Assumption (A5).}
$v \in C^{1,2}((0,T)\times\Omega)$ and
for some constant $C\geq 0$:
\be
  \sup_{x\in\mO} \|v_t(.,x)\|_{C^{0,\delta}([0,T])} \leq C,
  \qquad
  \sup_{t\in (0,T)} \|v_{xx}(t,.)\|_{C^{0,\delta}(\bar\mO)} \leq C.
\ee


\begin{rem}
\label{rem:regular}
By results in \cite{eva-len-81} and \cite{krylov1982boundedly},
assumption (A5) is satisfied 
for sufficiently smooth data and given a uniform ellipticity condition.
\end{rem}
We have the following error estimates:

\begin{theorem}\label{th:errors}
We assume (A1), (A2), (A3), and the CFL condition \eqref{eq:CFL}.
\begin{enumerate}
\item[$(i)$]
If 
{$v\in C^{3,4}_b((0,T)\times \mO)$,} then
\begin{eqnarray*}
\max_{0\leq k\leq N} |v^k-u^k|_0 \le
  C h^2,
\end{eqnarray*}
where $C$ is a constant which depends on the derivatives of $v$ of order 3 and 4 in $t$ and $x$, respectively.
\item[$(ii)$]
If (A5) holds for some $\delta \in (0,1]$, then 
the numerical solution $u$ of (\ref{eq:BDF2}), (\ref{eq:IEscheme}) converges to $v$ in the $L^2$-norm with
\begin{eqnarray*}
  \max_{0\leq k\leq N}|v^k-u^k|_0 \le
  C h^{\delta},
\end{eqnarray*}
for some constant $C$ (possibly different from the one in (A5)).
\end{enumerate}
\end{theorem}


\noindent
\begin{proof}
We first prove $(ii)$.
By Taylor expansion, we can write for instance, for some $\theta_1,\theta_2 \in [0,1]$,
\begin{eqnarray*}
\left\vert v_t(t,x) - \frac{v(t,x) - v(t-\tau,x)}{\tau} \right\vert \ = \ \left\vert v_t(t,x) - v_t(t-\theta_1 \tau ,x) \right\vert
&\le& C \tau^\delta
\end{eqnarray*}
and
\begin{eqnarray*}
 \left\vert v_t(t,x) - \frac{3 v(t,x) - 4 v(t-\tau,x) + v(t-2\tau,x)}{2 \tau} \right\vert && \\
 &&\hspace{-5 cm} \le \ \left\vert v_t(t,x) - \frac{1}{2} \left( 3 v_t(t-\theta_1 \tau ,x) - v_t(t-(1+\theta_2)\tau ,x) \right) \right\vert \\
 &&\hspace{-5 cm} \le \ \left\vert v_t(t,x) - v_t(t-\theta_1 \tau ,x) \right\vert
 + \frac{1}{2} \left\vert v_t(t-\theta_1 \tau ,x) - v_t(t-(1+\theta_2)\tau ,x) \right\vert \\
 &&\hspace{-5 cm} \le \  C \tau^\delta + \fud C (2\tau)^\delta \leq 2 C \tau^\delta.
\end{eqnarray*}
Similarly, using the higher spatial regularity, there exists a constant $C_0\geq 0$ such that 
\begin{eqnarray*}
\left\vert v_x(t,x) - \frac{3 v(t,x) - 4 v(t,x-h) + v(t,x-2h)}{2 h} \right\vert &\le& C_0 C h^{\delta +1}, \\
\left\vert v_{xx}(t,x) \, - \ \frac{v(t,x+h) - 2 v(t,x) + v(t,x-h)}{h^2} \right\vert &\le& C_0 C h^{\delta}.
\end{eqnarray*}
The result $(ii)$ now follows directly by inserting the obtained truncation error into the stability estimate 
of Theorem~\ref{th:L2stab}. 

For the proof of $(i)$ (smooth case),
expansion up to order 3 and 4 gives the truncation error of higher order for $k\ge 2$, 
and we use the fact that the error from the first backward Euler step is bounded by
$\|E^1\| \le C \tau (\tau + h^2)$; in particular, we use that $(E^1-E^0)/\tau + (\Delta^1A + F^1 B +R^1) E^1 = - \Eps^1$, with $\| \Eps^1\|\leq C(\tau + h^2)$, $E^0=0$ 
and the bound is otherwise similar and simpler than that for $k\geq 2$.
\end{proof}

The previous arguments can also be used to derive error estimates for piecewise smooth solutions.  
In this case, we will need to limit the number of non-regular points that may appear in the exact solution
(assumption (A6)$(i)$ is similar to \cite{bokanowski-debrabant}).


\medskip
{\sc Assumption (A6).}
There exists an integer $p\geq 1$ and functions $(x^*_j(t))_{1\leq j\leq p}$ for $t\in[0,T]$, such that, with
$\mO^*_T:= (\mO\times (0,T))\backslash \bigcup_{1\leq j\leq p} \{(t,x^*_j(t)),\ t\in(0,T)\}$,
the following holds:
\begin{enumerate}
\item[$(i)$]
{$v \in C^{3,4}_b(\mO^*_T)$;}
\item[$(ii)$]
   $\forall j$, $t\conv x^*_j(t)$ is Lipschitz regular.
\end{enumerate}

\medskip

We give the following straightforward preliminary result without proof:

\begin{lemma}\label{lem:limited-singularities}
Assume (A6) and the CFL condition \eqref{eq:CFL}. Then for all $t$
$$
  \mbox{Card} \{ j, \  x \conv v(t,x)\ \mbox{not regular in $[x_{j-2},x_{j+2}]$}\} \leq 5p
$$
and
$$
  \mbox{Card} \{ j, \  \mt \conv v(\mt,x_j)\ \mbox{not regular in $[t-2\tau,t]$} \} \leq Cp.
$$
for some constant $C\geq 0$ independent of $\tau,h$
("not regular" meaning not $C^4$ in the first case and not $C^3$ in the second one).
\end{lemma}

Such a situation will be illustrated in the numerical example of Section~\ref{sec:num2}.

\begin{theorem}\label{th:errors_piecewise}
We assume (A1), (A2), (A3) and the CFL condition \eqref{eq:CFL}.
Let (A5) and (A6) hold for some $\delta \in (0,1]$, then 
the numerical solution $u$ of (\ref{eq:BDF2}), (\ref{eq:IEscheme}) converges to $v$ in the $L^2$-norm with
\begin{eqnarray*}
  \max_{2\leq k\leq N}|v^k-u^k|_0 \le
  C h^{1/2+\delta},
\end{eqnarray*}
where $C$ is a constant 
independent of $h$.
\end{theorem}
\begin{proof}
Let $\mathbb{I}^k$ be the (finite) set of indices $i$ such that $v$ is not regular in $ \{t_k\} \times (x_i-2h,x_i+2 h)
\cup (t_k-2\tau,t_k) \times \{x_i\}$.
Then
\beno
  |\mathcal{E}^k |_0^2 &  = &  \sum_{i\in \mathbb{I}} |\mathcal{E}_i^k|^2 h =  \sum_{i\in \mathbb{I}^k} |\mathcal{E}_i^k|^2 h
    + \sum_{i\in \mathbb{I}\backslash \mathbb{I}^k} |\mathcal{E}_i^k|^2 h \\
   & \leq &  C |\mathbb{I}^k|  (\tau^\delta + h^\delta)^2 h + C (\tau^2 + h^2)^2.
\eeno

We then use the fact that $|\mathbb{I}^k|\leq C$ for some (different) constant $C$ by Lemma~\ref{lem:limited-singularities}
and that $(\tau^2 + h^2)^2 = O(h^4) = O(h^{2+\delta})$, $\tau^\delta +h^\delta = O(h^\delta)$ by the CFL condition \eqref{eq:CFL}, 
in order to obtain the desired result.
\end{proof}

\begin{rem}
\begin{enumerate}[(i)]
\item
Similar results can be derived for errors in the $A$-norm, however derivatives of one order higher are required due to the derivative in the definition of the norm.
\item
The estimates in Theorem \ref{th:errors} are not always sharp, as symmetries and the smoothing behaviour of the scheme can result in higher order convergence.
We discuss such special cases for Examples 1 and 2 in Section~\ref{sec:num}, Remarks \ref{rem-ex1} and \ref{rem-ex2}, respectively.
\item
These error estimates can be compared with \cite{bokanowski-debrabant}, 
where an error bound of order $h^{1/2}$ was obtained
for diffusion problems with an obstacle term,
under the main assumption that $v_{xx}$ is a.e.\ bounded 
with a finite number of singularities
(instead of (A5)) .
In the present context it seems natural to assume the H{\"o}lder regularity
of $u_t$ and $u_{xx}$ coming from the ellipticity assumption (see Remark \ref{rem:regular}).
\end{enumerate}
\end{rem}



\section{Numerical tests}\label{sec:num}
We now compare the performance of the BDF2 scheme with other second order finite difference schemes on
 two examples.

\subsection{Test 1: Eikonal equation}\label{sec:num1}
The first example is based on a deterministic control problem ($\sigma\equiv 0$) 
and motivates the choice of the BDF2 approximation for the drift term in \eqref{eq:spaceBDF}, compared to the more classical
centered scheme~\eqref{eq:spaceC}.
We consider 
$$
\left\{
\begin{array}{ll}
v_t + |v_x| = 0 ,& x\in (-2,2), \; t\in (0,T), \\
v(0,x)= v_0(x), & x\in (-2,2),
\end{array}
\right.
$$
with $v_0(x)=\max(0, 1-x^2)^4$ and $T=0.2$.
The initial datum is shown in Figure \ref{fig:1} (dashed line). The exact solution is  
$$
v(t,x) = \min(v_0(x-t), v_0(x+t)).
$$
\begin{rem}
The Eikonal equation can be written as $v_t + \max_{a\in \{-1,1\}} (a v_x) = 0$
 in HJB form.
Note that our theoretical analysis does not cover this example, however,
since in the degenerate case assumption (A4) is required, which is not satisfied here.
\end{rem}

\begin{figure}
\centering
\includegraphics[width=0.49\columnwidth]{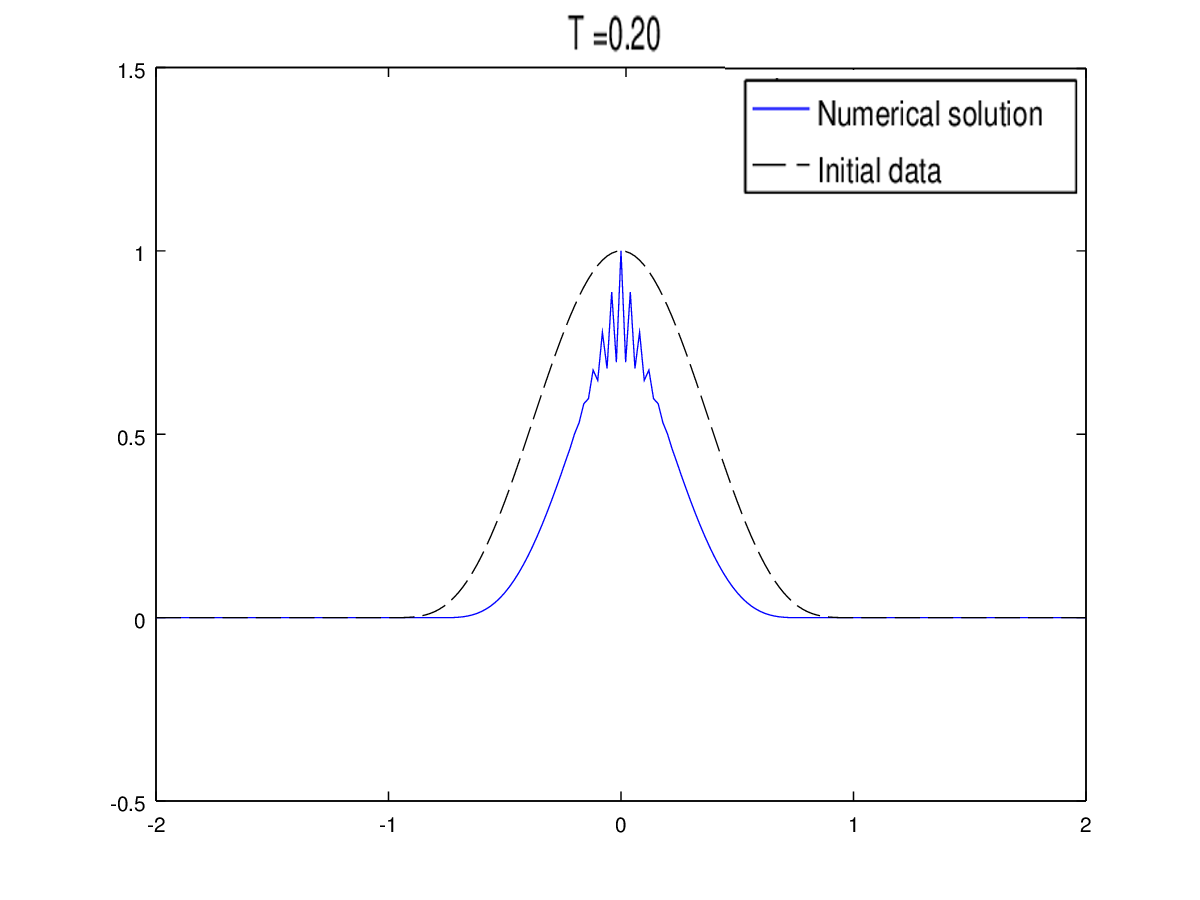}
\includegraphics[width=0.49\columnwidth]{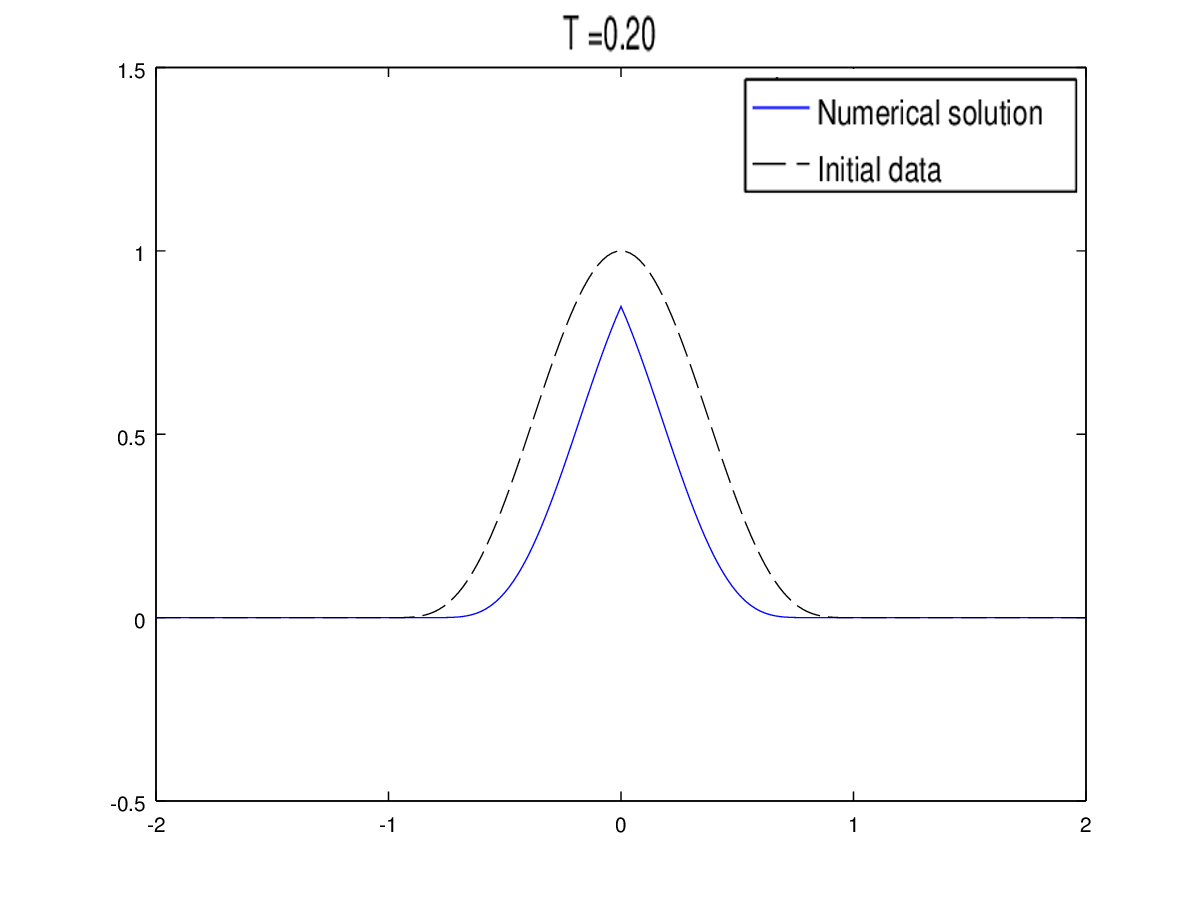} 
\caption{Test 1: Initial data (dashed line) and numerical solution at time $T=0.2$  computed for $I+1=200$ and $N=20$ ($\tau/h=0.5$)  using BDF in time and centred approximation of the drift (left), BDF in time and space (right).   
\label{fig:1}}
\end{figure}
In Figure \ref{fig:1}, we show the results obtained at the terminal time $T=0.2$ using schemes \eqref{eq:BDF2}-\eqref{eq:spaceC} (left) and \eqref{eq:BDF2}-\eqref{eq:spaceBDF} (right) with $\tau/h=0.5$.
We numerically observe that the centered approximation generates undesirable oscillations, whereas the BDF2 scheme is stable. 

{As stated in Theorem \ref{th:existence}, 
in case of a degenerate diffusion, a CFL condition of the form $\tau\leq Ch$  has to be satisfied for well-posedness of the BDF2 scheme.} 
Table~\ref{tab:Test1} shows numerical convergence of order $2$ in both time and space,
although the solution is globally only Lipschitz.
\begin{table}[h!]
\begin{tabular}{c|c|cc|cc|cc|c}
  $N$ & $I+1$
   & \multicolumn{2}{|c|}{$H^1$ norm} 
   & \multicolumn{2}{|c|}{$L^2$-norm} 
   & \multicolumn{2}{|c }{$L^\infty$ norm} &  CPU (s)\\
\hline\hline
 &      & error    & order & error    & order & error     & order & \\
 5 &   10  & 5.35E-01 &   -    & 1.25E-01 &   -    & 1.36E-01 &   -  & 0.094 \\ 
   10 &   20  & 2.42E-01 &  1.14 & 4.51E-02 &  1.47  & 6.83E-02 &  0.99 & 0.096 \\ 
   20 &   40  &  8.25E-02 &  1.55  & 1.55E-02 &  1.55  & 2.01E-02 &  1.77 &  0.126 \\ 
   40 &   80  & 2.38E-02 &  1.80  & 4.32E-03 &  1.84  & 5.23E-03 &  1.94 & 0.147\\ 
   80 &  160  & 6.26E-03 &  1.92  & 1.11E-03 &  1.96  & 1.31E-03 &  2.00 & 0.194\\ 
  160 &  320  & 1.61E-03 &  1.96  & 2.79E-04 &  1.99  & 3.24E-04 &  2.01 & 0.335\\ 
  320 &  640  & 4.09E-04 &  1.98  & 7.10E-05 &  1.99  & 8.19E-05 &  2.00 &  0.759\\ 
    640 & 1280  & 1.03E-04 &  1.99  & 1.78E-05 &  2.00  & 2.05E-05 &  2.00 & 2.306 \\ 
\end{tabular}
\caption{Test 1. Error and convergence rate to the exact solution for the BDF2 scheme with $\tau/h=0.1$ and initial data $v_0(x)=\max(0, 1-x^2)^4$.}\label{tab:Test1}
\end{table}

{
\begin{rem}
\label{rem-ex1}
The full convergence order here is due to the particular symmetry of the solution.
{
To confirm this, we report in Table \ref{tab:Test1_2} the results obtained for the same equation with initial data 
$$
v(0,x)=-\max(0, 1-x^2)^4
$$ 
(see also Figure \ref{fig:test1_2}). In this case, there  is no such symmetry around the two singular points and as a result the
full convergence order is lost:
the scheme is globally only of order $1$ in the $H^1$ norm and roughly $1.5$ in the $L^2$ and $L^\infty$ norm.
\begin{figure}
\begin{minipage}{1\columnwidth}
\floatbox[{\capbeside\thisfloatsetup{capbesideposition={right,center},capbesidewidth=6cm}}]{figure}
{\caption{Test 1: Initial data (dashed line) 
$
v_0(x)=-\max(0, 1-x^2)^4
$
and numerical solution at time $T=0.2$  computed for $I+1=200$ and $N=20$ ($\tau/h=0.5$)  using the BDF2 scheme. The convergence rates for this example are reported in Table \ref{tab:Test1_2}.
\label{fig:test1_2}}}
{\hspace{-1 cm} \includegraphics[width=1.15\columnwidth]{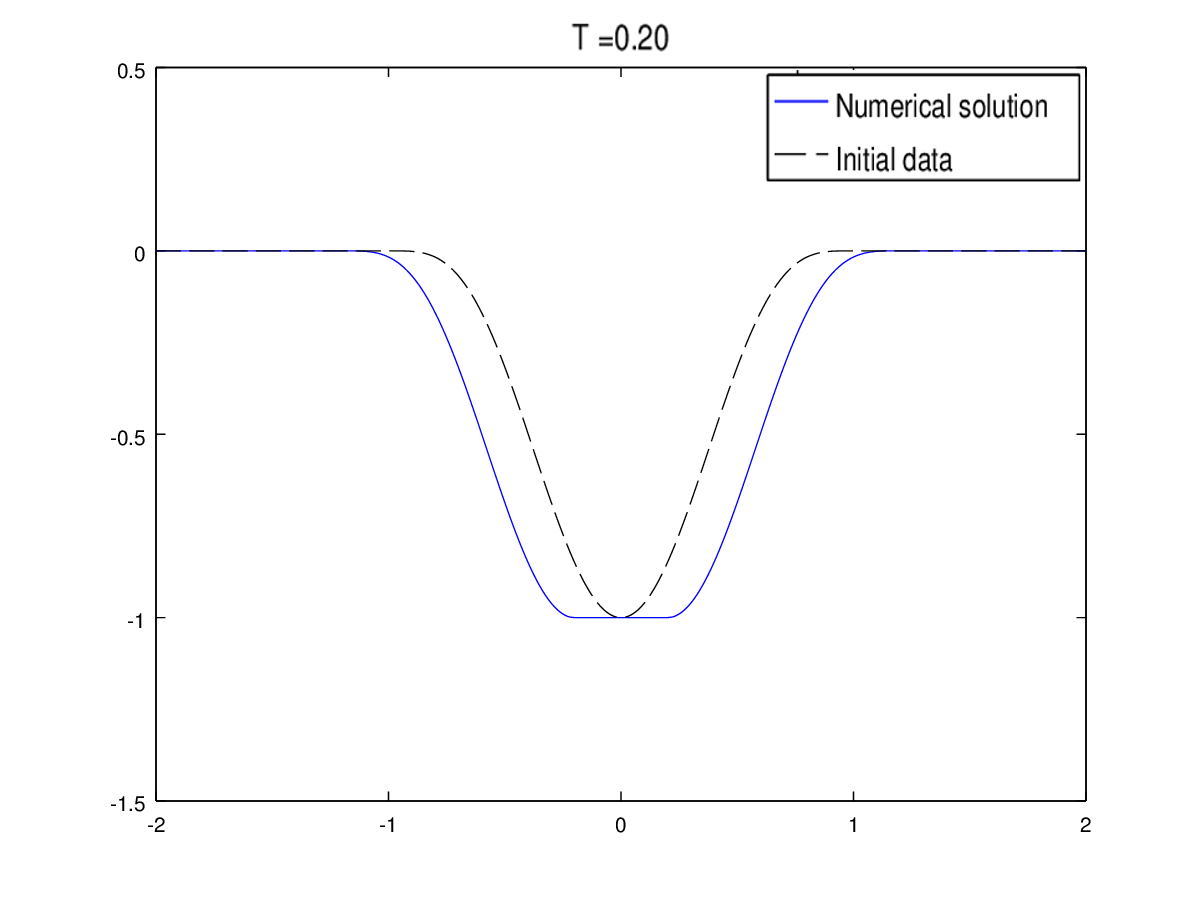}}
\end{minipage}
\end{figure}

\begin{table}[h!]
\begin{tabular}{c|c|cc|cc|cc|c}
  $N$ & $I+1$
   & \multicolumn{2}{|c|}{$H^1$ norm} 
   & \multicolumn{2}{|c|}{$L^2$ norm} 
   & \multicolumn{2}{|c }{$L^\infty$ norm} &  CPU (s)\\
\hline\hline
 &      & error    & order & error    & order & error     & order & \\
    5 &   10  & 5.84E-01 &   -    & 1.62E-01 &   -    & 1.51E-01 &   -   & 0.006\\ 
   10 &   20  & 2.69E-01 &  1.12  & 5.23E-02 &  1.63  & 6.20E-02 &  1.28  & 0.008\\ 
   20 &   40  & 1.45E-01 &  0.89  & 1.86E-02 &  1.49  & 2.08E-02 &  1.58  & 0.018\\ 
   40 &   80  & 6.74E-02 &  1.10  & 5.95E-03 &  1.64  & 7.89E-03 &  1.40  & 0.039\\ 
   80 &  160  & 3.20E-02 &  1.08  & 1.81E-03 &  1.72  & 3.57E-03 &  1.15  & 0.093\\ 
  160 &  320  & 1.60E-02 &  1.00  & 5.44E-04 &  1.73  & 1.51E-03 &  1.24  & 0.233\\ 
  320 &  640  & 8.16E-03 &  0.97  & 1.65E-04 &  1.72  & 6.33E-04 &  1.25  & 0.695\\ 
  640 & 1280  & 4.20E-03 &  0.96  & 5.09E-05 &  1.70  & 2.64E-04 &  1.26  & 2.163\\ 
\end{tabular}
\caption{Test 1. Error and convergence rate to the exact solution for the BDF2 scheme with $\tau/h=0.1$ and initial data $v_0(x)=-\max(0, 1-x^2)^4$.}\label{tab:Test1_2}
\end{table}
}
\end{rem}
}

\subsection{Test 2: A simple controlled diffusion model equation}\label{sec:num2}

The second test we propose is a problem with controlled diffusion.
We consider 
$$
\left\{
\begin{array}{ll}
v_t + \sup_{\sigma\in\{\sigma_1,\sigma_2\}}\Big(-\frac{1}{2}\sigma^2v_{xx}\Big) = 0, & x\in (-1,1), t\in (0,T), \\
v(0,x)= \sin(\pi x), & x\in (-1,1),
\end{array}
\right.
$$
with parameters $\sigma_{1}=0.1$, $\sigma_{2}=0.5$, $T=0.5$. 


In spite of the apparent simplicity of the equation under consideration, in \cite{pooley2003numerical} an example of non-convergence of the Crank-Nicolson scheme is given for a similar optimal control problem.
The BDF2 scheme, in contrast, has shown good performance for that same problem in
\cite{bokanowski2016high}.

Figure \ref{fig:2} (top row) shows the initial data and the value function at terminal time computed using the BDF2 scheme. The error and convergence rate in different norms are reported in Table \ref{tab:Test2_highCFL}. Here an accurate numerical solution computed by an implicit Euler scheme (in order to ensure convergence)  is used for comparison.
\begin{figure}
\centering
\includegraphics[width=0.47\columnwidth]{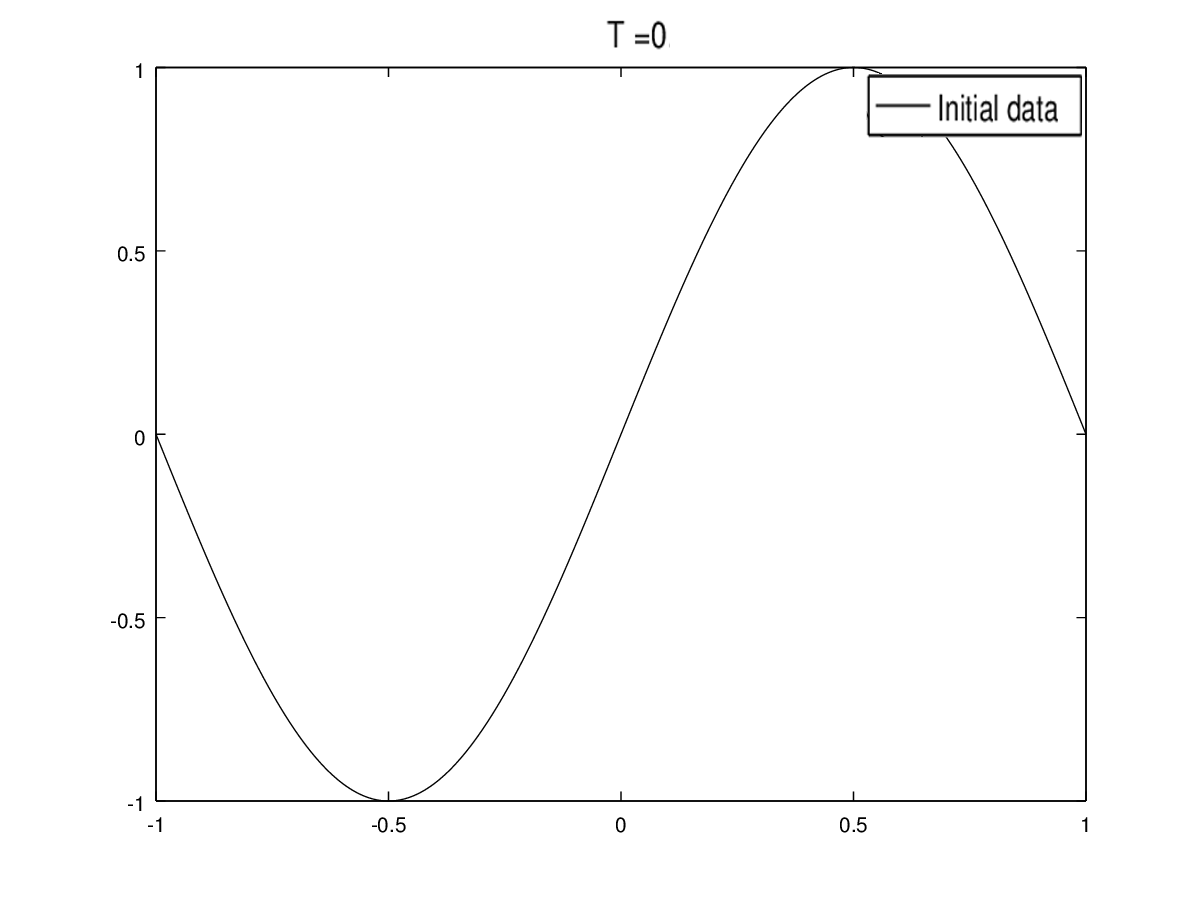}
\includegraphics[width=0.47\columnwidth]{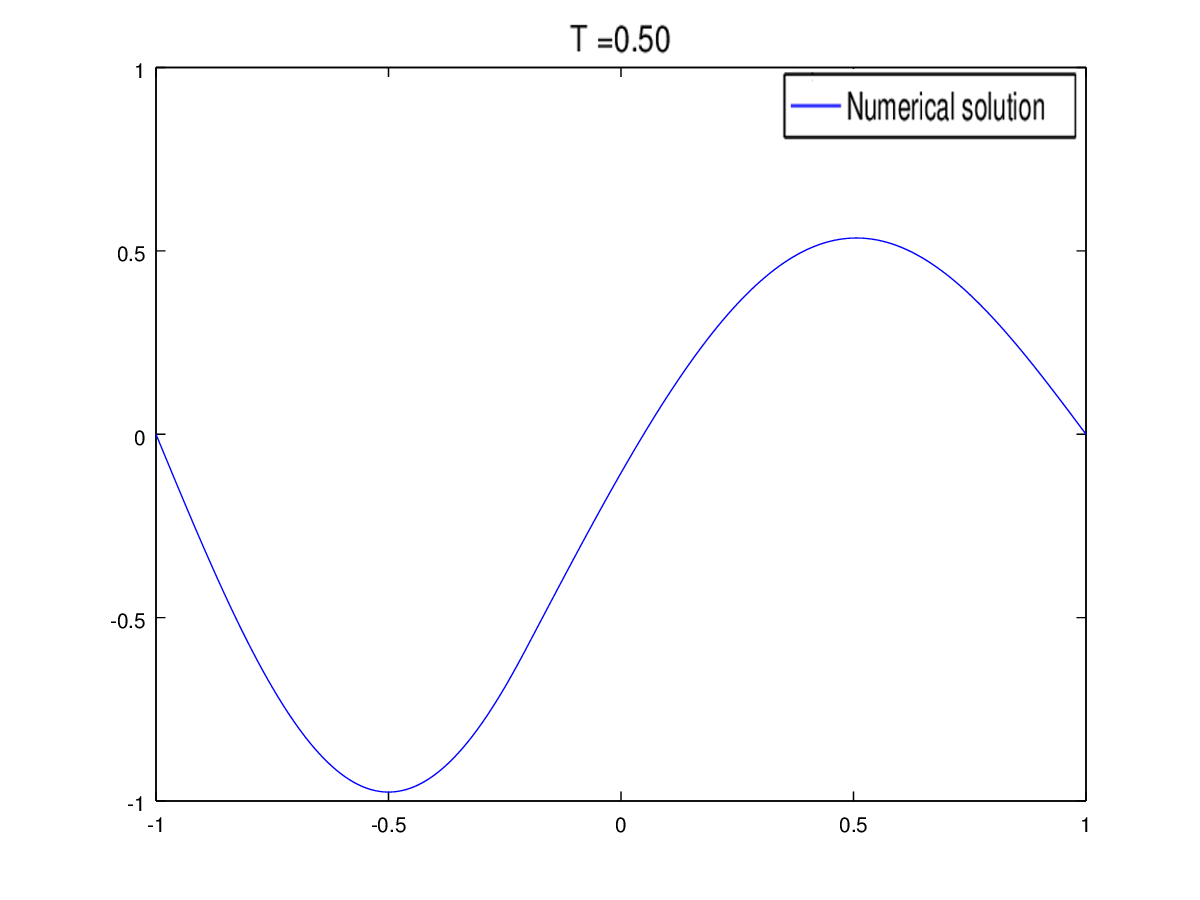} \\
\includegraphics[width=0.47\columnwidth]{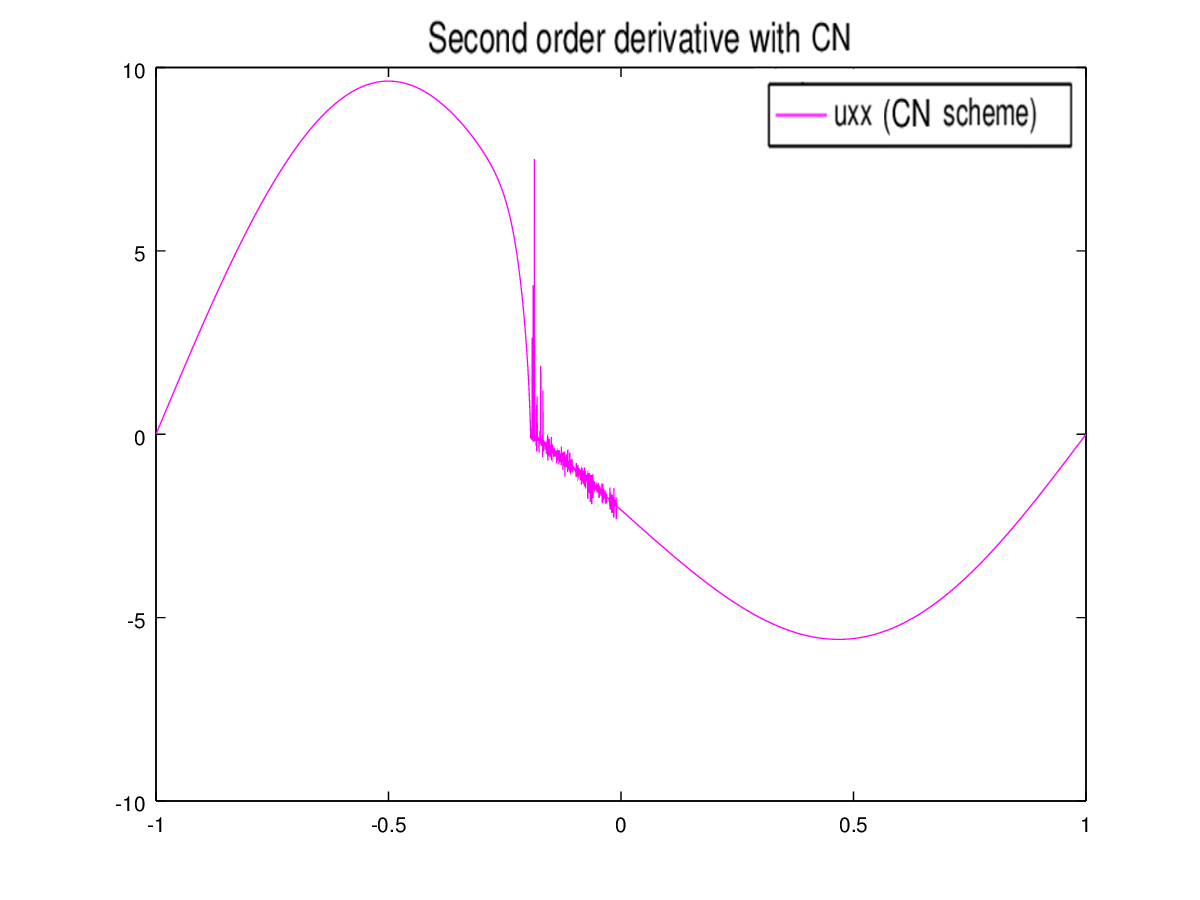}
  \includegraphics[width=0.47\columnwidth]{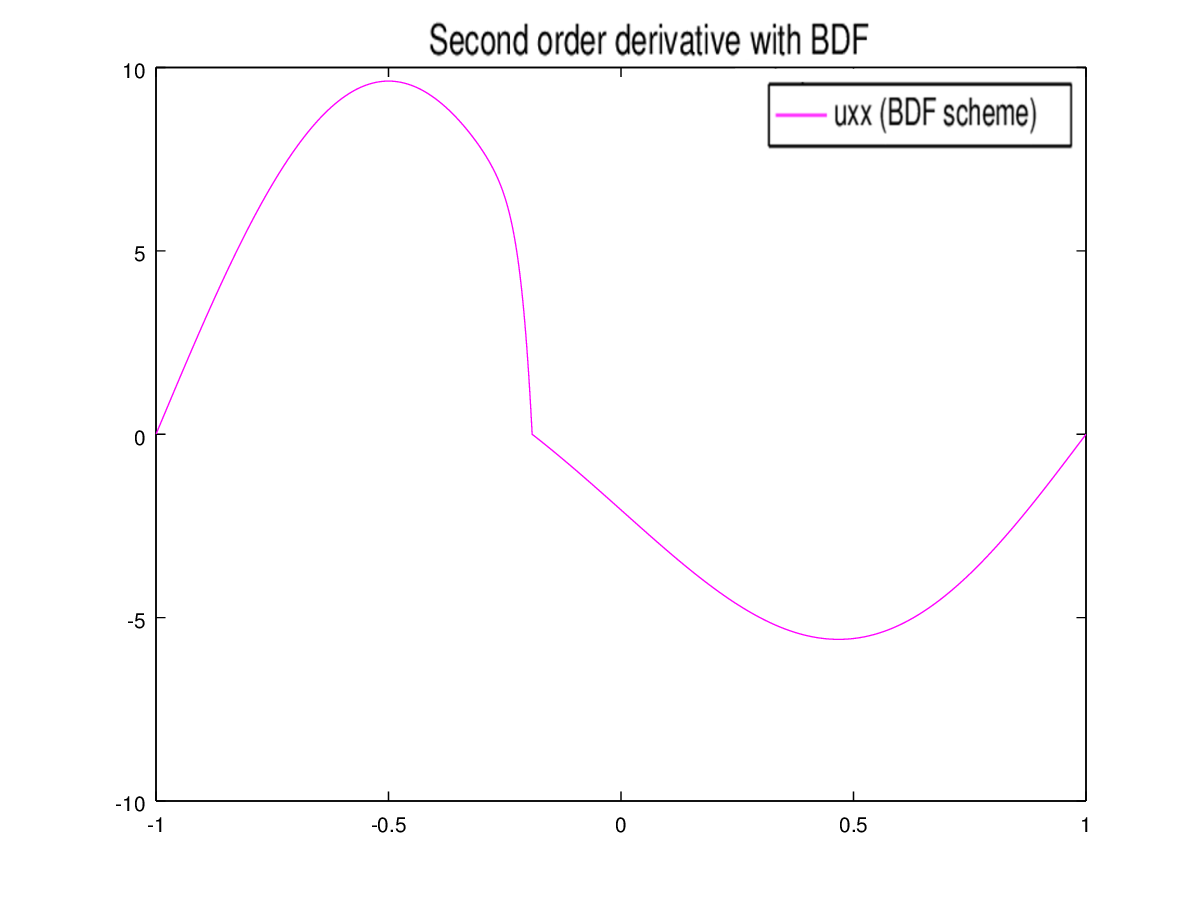} 
\caption{Test 2: Initial data (top, left), numerical solution at time $T=0.5$ (top, right) computed by the BDF2 scheme, second order derivative computed with CN scheme (bottom, left) and BDF2 (bottom, right) for $N=256$ and $I+1=5120$.   
\label{fig:2}}
\end{figure}

\begin{table}[h!]
\begin{tabular}{c|c|cc|cc|cc|c}
  $N$ & $I+1$
   & \multicolumn{2}{|c|}{$H^1$ norm} 
   & \multicolumn{2}{|c|}{$L^2$ norm} 
   & \multicolumn{2}{|c }{$L^\infty$ norm} & CPU (s) \\
\hline\hline
 &      & error    & order & error    & order & error     & order  \\
    1 &   20  &  1.54E-01 &   -     & 5.11E-02 &   -    & 7.24E-02 &   -    &  0.131\\ 
    2 &   40  &  5.53E-02 &  1.48  & 1.88E-02 &  1.45  & 2.63E-02 &  1.46 & 0.112 \\ 
    4 &   80  &  1.47E-02 &  1.91  & 5.17E-03 &  1.86  & 6.99E-03 &  1.91  & 0.111\\ 
    8 &  160  &  3.59E-03 &  2.04  & 1.27E-03 &  2.03  & 1.66E-03 &  2.08 &  0.122 \\ 
   16 &  320  &  8.98E-04 &  2.00   & 3.14E-04 &  2.02  & 4.09E-04 &  2.02  & 0.146\\ 
   32 &  640  &  2.26E-04 &  1.99  &  7.84E-05 &  2.00  & 1.02E-04 &  2.00  & 0.183\\ 
   64 & 1280  & 5.65E-05 &  2.00   & 1.96E-05 &  2.00  & 2.56E-05 &  2.00  &  0.267\\ 
  128 & 2560  & 1.42E-05 &  2.00 & 4.90E-06 &  2.00  & 6.42E-06 &  2.00  & 0.598 \\
   256 & 5120  & 1.21E-06 &  2.01  & 1.21E-06 &  2.01  & 1.59E-06 &  2.01  & 1.879\\ 
  
\end{tabular}

\caption{
Test 2. Error and convergence rate for the BDF2 scheme with high CFL number $\tau = 5 h$. A reference solution computed by the implicit Euler scheme \eqref{eq:BDF2-step1} with $I+1=20\times 2^{9}, N = 2^{22}$ is used.
}\label{tab:Test2_highCFL}
\end{table}

\begin{table}[h!]
\begin{tabular}{c|c|cc|cc|cc|c}
  $N$ & $I+1$
   & \multicolumn{2}{|c|}{$H^1$ norm} 
   & \multicolumn{2}{|c|}{$L^2$ norm} 
   & \multicolumn{2}{|c }{$L^\infty$ norm} & CPU (s) \\
\hline\hline
 &      & error    & order & error    & order & error     & order  \\
  1 &   20  &  4.11E-02 &   -    & 7.01E-03 &   -    & 9.44E-03 &   -   &  0.149 \\ 
 2 &   40  &  7.82E-03 &  2.39  & 1.45E-03 &  2.27  & 2.29E-03 &  2.04 & 0.113  \\ 
4 &   80  &  1.97E-03 &  1.99  & 3.87E-04 &  1.91  & 5.62E-04 &  2.03 & 0.111 \\ 
8 &  160  &  5.16E-04 &  1.94  & 1.02E-04 &  1.92  & 1.45E-04 &  1.95 & 0.128 \\ 
16 &  320  &  1.09E-04 &  2.24  & 2.67E-05 &  1.94  & 3.77E-05 &  1.95  & 0.166\\ 
32 &  640  & 2.96E-05 &  1.88  & 7.15E-06 &  1.90  & 9.87E-06 &  1.93 & 0.188 \\ 
64 & 1280  & 7.64E-06 &  1.96  & 2.03E-06 &  1.82  & 2.61E-06 &  1.92  & 0.310\\ 
128 & 2560   & 9.50E-05 & -3.64  & 1.98E-05 & -3.29  & 3.49E-05 & -3.74  & 0.992\\ 
256 & 5120   & 7.18E-04 & -2.92  & 8.40E-05 & -2.08  & 1.62E-04 & -2.22  &  4.251
    \end{tabular}
  \caption{
  Test 2. Error and convergence rate for the CN scheme with high CFL number $\tau = 5 h$. A reference solution computed by the implicit Euler scheme \eqref{eq:BDF2-step1} with $I+1=20\times 2^{9}, N = 2^{22}$ is used.}\label{tab:Test2_CNhighCFL}
  \end{table}

{
Taking $\tau\sim h$ the BDF2 scheme gives clear second order convergence, see Table \ref{tab:Test2_highCFL}. This is not the case for CN as shown in  Table \ref{tab:Test2_CNhighCFL}.
The CN scheme also exhibits 
some  instability in the second order derivative for high CFL number, i.e.\ $\tau/h$, see Figure \ref{fig:2} (this is analogous to the finding in \cite{pooley2003numerical}). One can verify that for a small CFL number, i.e.\ $\tau\sim h^2$, the CN scheme shows second order of convergence.
}

{

\begin{rem}
\label{rem-ex2}
In this example, due to the strict ellipticity, Assumption (A5) is guaranteed for some $\delta>0$ (see Remark \ref{rem:regular}).
Then Theorem \ref{th:errors}
gives convergence with order $\delta$.
Furthermore, Fig.~\ref{fig:2}, bottom row, suggests H{\"older} continuity of $u_{xx}$ in $x$,
which is expected by virtue of the control being piecewise constant.
Therefore, we conjecture that Assumption (A6) is satisfied, such that Theorem \ref{th:errors_piecewise} would give the higher order $1/2+\delta$.
In the test, 
in fact the full order 2 is observed (see Table \ref{tab:Test2_highCFL}).


\end{rem}

}

\section{Conclusion}  \label{sec:conclusion}
 
 We have proved the well-posedness and stability in $L^2$ and $H^1$ norms of a second order  BDF scheme for HJB equations
 with enough regularity of the coefficients.
 The significance of the results is that this was achieved for a second order (and hence) non-monotone scheme.
 For smooth or piecewise smooth solutions, as is often the case, one can use the recursion we derived to bound the error of the numerical solution in terms of the truncation error of the scheme.
 The latter depends on the regularity of the solution and has to be estimated for individual examples.
 
 The numerical tests demonstrate convergence at least as good as predicted by the theoretical results, and often better, due to symmetries of the solution or smoothing properties of
 the equation and the scheme. This is in contrast to some alternative second order schemes, such as the central spatial difference in the case of a first order equation, or the Crank-Nicolson time stepping scheme for a second order equation, which can show poor or no convergence.

\appendix


\section{
Proof of Lemma~\ref{lem:exist}}\label{app:A}
In order to prove the existence and uniqueness of a solution to \eqref{eq:lem-1},
we consider a fixed-point approach.
The initial problem \eqref{eq:lem-1} can be written as follows:
\be\label{eq:lemexists-1}
  \sup_{a\in \Lambda} (L_a X - (q_a-U_a X)) = 0,
\ee
where $L_a$ and $U_a$ are two matrices such that $M_a \equiv L_a + U_a$. 
We consider in particular $L_a$ to be the lower triangular part of $M_a$ including the diagonal terms, $(L_a)_{ij}:=(M_a)_{ij} 1_{i\geq j}$,
and $U_a$ the remaining upper triangular part, $(U_a)_{ij}:=(M_a)_{ij} 1_{i<j}$.

For a given vector $c\in \R^I$, let $g(c):=X$ denote the (unique) 
solution of the following simplified problem:
\be\label{eq:lemexists-2}
  \sup_{a\in \Lambda} (L_a X - (q_a-U_a c)) = 0.
\ee
Indeed, because $(L_a)_{ii}=(M_a)_{ii}>0$, denoting $v_a:=q_a - U_a c$,  it is easy to see by recursion in $i$ that the unique solution of 
$$ 
  \sup_{a\in \Lambda} (L_a X - v_a) = 0
$$
is given by 
$$
   x_i:=\inf_{a\in \Lambda} \bigg({\Big(}(v_a)_i - \sum_{k=1}^{i-1} (L_a)_{ik} x_k{\Big)} / (L_a)_{ii} \bigg).
$$
Therefore, solving \eqref{eq:lemexists-1} amounts to solving $g(X)=X$.
{By elementary compuations one can show that}  
$g$ is $\delta$-Lipschitz for the $\|.\|_\infty$ norm, with $\delta:=\sup_a \| (L_a)^{-1} U_a\|_\infty$.

For a diagonally dominant matrix, the following classical estimate holds
$$
  \| (L_a)^{-1} U_a\|_\infty \leq \sup_{i\in \mathbb{I}} \frac{\sum_{j>i} |(M_a)_{ij}|}{|(M_a)_{ii}| - \sum_{j<i} |(M_a)_{ij}|}
$$
(this is related to the Gauss-Seidel relaxation method; see for instance, Th.\,8.2.12 in~\cite{sto-bul-1993}).
By using the assumptions on the matrices $M_a$, we have $\delta<1$. 
Hence, $g$ is a contraction mapping on $\R^I$ 
and therefore we obtain the existence and uniqueness of a solution of \eqref{eq:lemexists-1} as desired.
$\Box$


 \bibliographystyle{plain}
 \bibliography{biblio}



\end{document}